\documentclass[10pt,a4paper]{amsart}

\usepackage{amsmath,amssymb,latexsym,amsthm,enumerate}
\usepackage{verbatim}
\usepackage{graphicx} 
\usepackage{hyperref} 
\usepackage{xcolor} 
\usepackage{psfrag} 
\usepackage{subfigure} 
\usepackage{listings}
\usepackage{here}
\usepackage{float,multirow}
\floatstyle{ruled}
\newfloat{algo}{thp}{lop}
\floatname{algo}{Algorithm}

\def\E{\mathbb{E}}

\def\P{\mathbb{P}}
\def\PP{\mathbb{P}}

\def\P{\mathbf{P}}

\def\1{\mathbf{1}}
\def\0{\mathbf{0}}

\newtheorem{Thm}{Theorem}[section]
\newtheorem{Cor}[Thm]{Corollary} 
\newtheorem{Lemma}[Thm]{Lemma}

\newtheorem{Prop}[Thm]{Proposition} 

\newtheorem{Def}[Thm]{Definition}

\newtheorem{As}[Thm]{Assumption}
\newtheorem{rems}[Thm]{Remark}

\def\BEN{\begin{enumerate}}  \def\BI{\begin{itemize}}
\def\EEN{\end{enumerate}}   \def\EI{\end{itemize}}

\def\mbb{\mathbb} \def\mbf{\mathbf} 
\def\mc{\mathcal} \def\unl{\underline} \def\ovl{\overline}

\def\le{\left}
\def\ri{\right}
\def\te#1{\mathrm{e}^{#1}}   
\def\WT{\widetilde}
\def\WH{\widehat}

\def\t{\tau}     
\def\w{\omega} \def\q{\qquad} 
   
  \def\td{\text{\rm d}}
 
\numberwithin{equation}{section}

\begin{document}

\title[Randomisation and recursion methods for mixed-exponential L\'{e}vy models]%
{Randomisation and recursion methods for \\mixed-exponential L\'{e}vy models, 
with financial applications}  
 
\address{Department of Mathematics, Imperial College London}
\author{Aleksandar Mijatovi\'{c}, Martijn Pistorius, Johannes Stolte}
\email{\{a.mijatovic, m.pistorius, j.stolte09\}@imperial.ac.uk}
\date{\today\newline
 {{{\it Acknowledgements.}
We thank the Editor and an anonymous referee, and Dan Crisan, Lane Hughston, Antoine Jacquier, Felicity Pearce, Vladimir Piterbarg, Johannes Ruf, David Taylor and Josef Teichmann, the participants of the Global Derivatives Trading \& Risk Management - Barcelona (2012), the Seventh World Congress of the Bachelier Finance Society - Sydney (2012), the Finance and Stochastics Seminar at Imperial College London (2011), and a satellite workshop at the University of the Witwatersrand - Johannesburg (2011), for useful comments. 
JS was supported by an EPSRC DTA grant and
a doctoral grant (Ref. Nr. D/11/42213) 
from the German Academic Exchange Service (DAAD)}}}

\begin{abstract}
We develop a new Monte Carlo variance reduction 
method to estimate the expectation of 
two commonly encountered path-dependent functionals: 
first-passage times and occupation times of sets. 
The method is based on a recursive approximation of the first-passage 
time probability and expected occupation time 
of sets of a L\'{e}vy bridge process that 
relies in part on a randomisation of the time parameter.
We establish this recursion for general L\'{e}vy processes 
and derive its explicit form for mixed-exponential jump-diffusions,
a dense subclass (in the sense of weak approximation) of L\'evy processes, 
which includes Brownian motion with drift, 
Kou's double-exponential model and 
hyper-exponential jump-diffusion models. We present a highly 
accurate numerical realisation and derive error estimates. 
By way of illustration the method is applied to the valuation of range accruals and barrier options 
under exponential L\'{e}vy models and Bates-type stochastic volatility models with exponential jumps. 
Compared with standard Monte Carlo methods, we find that the method 
is significantly more efficient.


\end{abstract}

\maketitle

\section{Introduction} 
{\bf Motivation and brief outline.} The Markov bridge sampling method for 
the estimation of the expectation 
 $\E \left[ F(T,\xi) \right]$ of a given path-functional $F$ of 
 a Markov process $\xi$ and the horizon $T>0$ 
 consists of averaging conditional 
 expectations $\WT F(\xi_{t_0}, \ldots, \xi_{t_N})$ 
 over $M$ independent copies
 $(\xi^{(i)}_{t_0}, \ldots, \xi^{(i)}_{t_N})$, $i=1, \ldots, M$, of 
 the values $(\xi_{t_0}, \ldots, \xi_{t_N})$ that
 $\xi$ takes on the grid 
 $\mbb T_N=\{0=t_0 < t_1 < \ldots < t_N = T\}$:
 \begin{equation}\label{eq:EfX}
 \E \left[ F(T,\xi) \right] \approx \frac{1}{M}\sum_{i=1}^M \WT F(\xi^{(i)}_{t_0}, \ldots, \xi^{(i)}_{t_N}), 
\end{equation}
where $\WT F(\xi_{t_0}, \ldots, \xi_{t_N})$ denotes the regular version of the conditional expectation 
$\E\le[F(T,\xi) | \xi_{t_0}, \ldots, \xi_{t_N}\ri]$. The name of the method 
derives from the fact that, conditional on the values  $(\xi_{t_0}, \ldots, \xi_{t_N})$, the stochastic processes $\{\xi_t, t\in[t_i, t_{i+1}]\}$, for $i=0, \ldots, N-1$, 
are equal in law to Markov bridge processes. 
The estimator in \eqref{eq:EfX} is unbiased and 
has strictly smaller variance than the standard Monte Carlo estimator, 
as a consequence of the tower property of conditional 
expectation and the conditional 
variance formula. The Markov bridge sampling method has the advantage that it allows for refinements of the generated path to the required level of accuracy, and can be combined with importance sampling. Such a bridge method is especially suited for the evaluation of expectations of 
path-dependent functionals 
(see \cite{Boyle1997}, for example). 
Since the function $\WT F$ is in general not available in closed 
or analytically tractable form, the viability of the Markov bridge method hinges on the ability to efficiently approximate the function~$\WT F.$ 
In this paper we derive an efficient approximation method for 
the conditional expectations $\WT F$ of certain path-dependent functionals 
given in terms of occupation times of sets and first-passage 
times, which is achieved by approximating the law of the bridge process by the law of the process pinned down at an independent random time with small variance. 
Since the latter law is analytically tractable when $\xi$ is a mixed-exponential L\'{e}vy process, this allows us to develop a Markov bridge Monte Carlo method for estimation of the corresponding expectation $\E[F(T,\xi)]$. 
To demonstrate the potential of the simulation method
we extend the approach to a two-dimensional Markovian setting, and deploy the method to numerically approximate 
the values of two common path-dependent derivatives, barrier options and range accruals, under a version of 
the Bates model~\cite{Bates1996}, which is an example of a 
stochastic volatility model with jumps that is widely used in financial modelling---we refer to \cite{Gatheral2006,Cont2004} for background. 

{\bf Literature overview.} 
In the literature~\cite{FT,Metwally2002,RufScherer}
a number of bridge sampling methods exist 
dealing with cases in which $\xi$ is a one-dimensional L\'{e}vy process. 
In \cite{FT} an adaptive 
bridge sampling method is developed for real-valued L\'{e}vy processes 
based on short-time asymptotics of stopped L\'{e}vy processes. 
By conditioning on the jump-skeleton and exploiting the explicit form of the distribution of the maximum of a Brownian bridge, 
a simulation method for pricing of barrier options under jump-diffusions  is presented in~\cite{Metwally2002}, and a refinement of this algorithm and application to the pricing of corporate bonds is 
given in~\cite{RufScherer}. An exact simulation algorithm for generation of diffusion sample paths 
deploying Brownian bridges is designed and analysed  in~\cite{Beskos2006}.

Several alternative methods have been developed for approximation
of path-dependent functionals, often based on weak or strong (pathwise) approximations of the solution of the SDE. In the setting of diffusions, a classical treatment of various strong and weak approximation schemes is given in~\cite{KlPl}. More recently, the problem of approximation of general 
path-dependent functionals has also received attention in 
the case of L\'{e}vy-driven SDEs. In~\cite{Dereich}  
a multi-level Monte Carlo algorithm is developed for path-dependent functionals of L\'{e}vy driven SDEs that are Lipschitz continuous in the supremum norm, and identifies error bounds. This algorithm is based on an approximation of the driving L\'{e}vy process by a L\'{e}vy jump-diffusion constructed by replacing the small jumps by a Brownian motion, as was investigated in~\cite{AR}. Adopting an alternative approach that does not rely on the Brownian small-jump approximation, a multi-level extension is presented in \cite{FKSS}  
of the Monte Carlo method developed in~\cite{Kuznetsov2010} for estimation of 
Lipschitz 
functions of the final value and running maximum 
of a real-valued L\'{e}vy process. Some functionals that are of interest in various applications
are not included in the analysis of \cite{Dereich,FKSS}, as these fail to satisfy  the Lipschitz condition.
The bridge method that we present in the current paper 
provides approximations in two such cases, namely, the distribution of the running maximum 
and the expected occupation time of sets.

 {\bf Approximation of bridge functionals.}  
As mentioned above, a key-step in the development of the Markov bridge method is the availability of an efficient approximation of the 
conditional expectations $\WT F$. As in general the transition probabilities of the Markov processes considered here are not explicitly available, the first step is to approximate the Markov process in question by its continuous-time Euler-Maruyama (EM) scheme. The approximation of expectations of path-dependent
functionals under stochastic volatility models with jumps
using the continuous-time EM-scheme is based on the 
{\em harness property} of a Markov process
which states that, for any two epochs $t_1$ and $t_2$ the collections of values of the Markov process 
at times in between $t_1$ and $t_2$ is independent of 
the values for $t$ outside this interval, 
conditional on the values of the process at $t_1$ and $t_2$. 
Noting that a L\'{e}vy process that is conditioned to start 
from position $x$ and to take the value $y$ at the horizon $T$
is equal in law to a L\'{e}vy bridge process from $(0,x)$ to $(T,y)$, we are led to 
the problem of evaluating the expectations of path-dependent functionals of L\'{e}vy bridges.

{\bf Randomisation method and recursions.} The approximation method of the L\'{e}vy 
bridge quantities that we present is based in part 
on a randomisation of the time-parameter. This randomisation method was originally developed in \cite{Carr1998} for the 
valuation of American put options, and is known as Erlangisation in risk theory 
\cite[Ch. IX.8]{AA}. The method has been deployed in~\cite{AAU} for the efficient computation of ruin probabilities 
and in~\cite{ACU,BLMF,KvS,Kuznetsov2010,KP,Levconv} for the valuation of American-type and barrier options. 
This randomisation method is based 
on the fact that, according to the law of large numbers, the average of 
independent exponential random variables with mean $t$ 
converges to $t$. An average of $n$ such exponential random variables 
is equal in distribution to a Gamma$(n,n/t)$ random variable 
$\Gamma_{n,n/t}$, which has mean 
$t$ and variance $t^2/n$. 
As observed in~\cite[Ch. VII.6]{Feller}, 
the approximation of the value $f(t)$ of a continuous bounded function $f$ at $t>0$ by the expectation $\E[f({\Gamma_{n,n/t}})]$ of
$f$ evaluated at the random time $\Gamma_{n,n/t}$ 
is asymptotically exact: since $\Gamma_{n,n/t}$ converges to a point mass at $t$, it follows that the expectation $\E[f({\Gamma_{n,n/t}})]$ converges to $f(t)$ 
as $n$ tends to infinity. 
As regards the rate of convergence,
the form of the PDF of $\Gamma_{n,n/t}$ implies that, in the case 
that $f$ is $C^2$ at $t$, 
the decay of the error $\E[f({\Gamma_{n,n/t}})] - f(t)$ is linear in $1/n$, in line with \cite[Theorem 6]{AAU}, and that, moreover, $\E[f(\Gamma_{n,n/t})]$ 
admits the following expansion 
if the function $f$ is $C^{2k}$ at $t$:
\begin{equation*}
\E[f(\Gamma_{n,n/t})] - f(t) 
= \sum_{m=1}^{k} b_m(t)\le(\frac{1}{n}\ri)^m + o(n^{-k})\q\text{as $n\to\infty$},
\end{equation*}  
\noindent for certain functions $b_1, \ldots, b_k$ (given in Theorem~\ref{lem:conv} 
below). We apply this expansion to functions $f(t)$ that are equal to the expectations of 
path-dependent functionals of L\'{e}vy bridges living on the time-interval $[0,t]$. We note that $\E[f({\Gamma_{n,n/t}})]$ is equal to the 
expectation of the corresponding path-functional of the L\'{e}vy process $X$ pinned down 
at an independent random time that is equal in distribution to $\Gamma_{n,n/t}$.
For the path-dependent functionals that we consider (namely, first-passage times and occupation times of sets) 
the corresponding functions $f$ are sufficiently smooth, so that the use of the Richardson 
extrapolation is fully justified. It holds furthermore (see Theorem~\ref{thm:rec}) that 
the density functions $D_n(x,y)$ and $\Omega_n(x,y)$, $n\in\mbb N,$ given by
$D_{n,q}(x,y)\td y = \PP(\ovl X_{\Gamma_{n,q}}\leq x, X_{\Gamma_{n,q}}\in\td y)$ 
and  $\Omega_{n,q}(x,y)\td x\,\td y = \E\le[\int_0^{\Gamma_{n,q}}I_{\{X_u \in\td x , X_{\Gamma_{n,q}} \in\td y\}}\,
\td u\,\ri]$
corresponding to a random horizon $\Gamma_{n,n/t}$ 
satisfy the following recursions for $x,y\in\mbb R$ and $n\in\mbb N$: 
\begin{eqnarray}\label{eq:Drec}
&&  D_{n+1,q}(x,y) = \int_{-\infty}^x D_{n,q}(x-w,y-w) D_{1,q}(x,w) \td w,\q \max\{y,0\}\leq x,\\
\label{eq:Orec}
&& \Omega_{n+1,q}(x,y) = \int_{-\infty}^\infty\le[\Omega_{1,q}(x,w) u_{n,q}(y-w) 
+ \Omega_{n,q}(x-w,y-w)u_{1,q}(w)\ri]\td w,
\end{eqnarray}
where $u_{n,q}$ is the probability density function of the random variable $X_{\Gamma_{n,q}}$. 
For the dense class of mixed-exponential L\'{e}vy processes 
(see Definition~\ref{def:mixep} below) we present explicit 
solutions to these recursions. By way of numerical illustration the method 
was implemented for a number of models in this class, 
and the numerical outcomes are reported in 
Section~\ref{sec:NumericalResults}, confirming the theoretically predicted rates of decay of the error. 
We observed that the Richardson extrapolation based on a small number %
(about ten) recursive steps already yields 
highly accurate approximations. 

{\bf Markov bridge method.} We combine subsequently 
these approximations with a continuous-time EM scheme
to estimate the conditional expectations $\WT F$ 
corresponding to the first-passage times and 
occupation times of sets of a stochastic volatility 
process with jumps. To illustrate the effectiveness of the method 
we evaluated a barrier option and a range note 
under a Bates-type model using the proposed 
Markov bridge Monte Carlo scheme, and report the results in Section~\ref{sec:option}.
The rates of decay of the error that we find numerically in the case of barrier options 
are in line with the corresponding error estimates 
that were established in~\cite{Gobet} for the case of killed diffusion processes.

{\bf Contents.} 
The remainder of this paper is organized as follows. 
In Section~\ref{sec:DefinitionOfProcessAndProperties} explicit expressions
are derived for the first-passage probabilities and expected occupation 
times of a mixed-exponential L\'{e}vy process. 
Section~\ref{sec:conv} is devoted to error estimates and numerical illustrations are presented in Section~\ref{sec:NumericalResults}. Section~\ref{sec:option} contains a Markov bridge sampling method based on the randomisation method and numerical illustrations. The proof of the recursions~\eqref{eq:Drec} and \eqref{eq:Orec} is deferred to 
Appendix~\ref{sec:BridgeSampling}.

\section{Maximum and occupation time of mixed-exponential L\'{e}vy models} 
\label{sec:DefinitionOfProcessAndProperties}
We show in this section that the recursions in \eqref{eq:Drec} and \eqref{eq:Orec} admit explicit solutions 
in the case that the L\'{e}vy process $X$ is a mixed-exponential jump-diffusion, the definition of which we recall next.

\begin{Def}\label{def:mixep}\rm 
(i) A random variable has a {\em mixed-exponential density} if it has PDF $f$ given by
\begin{eqnarray}\label{eq:mixeden}
&& f(x)=\sum_{i=1}^{m^+}p_i^+\alpha_i^+\te{-\alpha_i^+ x}I_{(0,\infty)}(x) +
\sum_{j=1}^{m^-}p_j^-\alpha_j^-\te{-\alpha_j^- |x|}I_{(-\infty,0)}(x), \ \text{where}\ \\
&& \sum_{k=1}^{m^\pm}p_k^\pm=q^\pm, \quad
q^++q^-=1 \quad \text{and} \quad
-\alpha_{m^-}^-<\cdots<- \alpha_1^-<0<\alpha_1^+<\cdots< \alpha_{m^+}^+.\nonumber
\end{eqnarray}

(ii) A L\'{e}vy process $X=\{X_t, t\in\mbb R_+\}$ is a {\em mixed-exponential jump-diffusion} (MEJD)
if it is of the form
\begin{eqnarray}
\label{eq:Process}
X_t & = & \mu t+ \sigma W_t +\sum_{i=1}^{N_t}U_i, 
\end{eqnarray}
where $\mu$ is a real number and $\sigma$ is strictly positive, $ W $ is a standard Brownian motion, 
$N$ is a Poisson process with intensity $\lambda$, and the jump-sizes
$\{U_i, i\in\mbb N\}$
are IID with mixed-exponential density.
Here, the collections $W=\{W_t, t\in\mbb R_+\}$, 
$N=\{N_t, t\in\mbb R_+\}$ 
and $\{U_i, i\in\mbb N\}$ 
are independent.
\end{Def}
\begin{rems}\rm
(i) Including in Def.~\ref{def:mixep} the additional restriction that the weights $p_k^\pm$ are nonnegative, the 
L\'{e}vy process is a hyper-exponential jump-diffusion (HEJD). 
While HEJD processes are dense in the class of all L\'{e}vy processes with a completely monotone L\'{e}vy density, the collection of mixed-exponential jump-diffusions is dense in the class of all L\'{e}vy processes, in the sense of weak convergence of probability measures
(see \cite{Botta1986}).

(ii) The parameters $\{p_k^\pm, k=1, \ldots, m^\pm\}$ 
cannot be chosen arbitrarily but need to satisfy a restriction 
to guarantee that $f$ is a PDF. 
Necessary and sufficient conditions for $f$ to be a PDF are 
$$ p_1^\pm > 0, \qquad \sum_{k=1}^{m^\pm}p_k^\pm \alpha_k^\pm \geq 0, 
\q\text{and}\q \forall l = 1,...,m^\pm: \q\sum_{k=1}^{l}p_k^\pm \alpha_k^\pm \geq 0, 
$$ 
respectively. 
For a proof of these results and alternative conditions 
see \cite{Bartholomew1969}. 
In Section~\ref{sec:option} we will impose  the additional condition $ \alpha_1^+ > 1$, which ensures that the expectation $\E[S_t]$ of the exponential L\'{e}vy process $S_t = \exp\{X_t\}$ is finite for any non-negative $t$.

(iii) Samples can be drawn from the mixed-exponential distribution by using the acceptance-rejection method (see\cite{Press2002}) and taking as the instrumental distribution a double-exponential distribution. The double-exponential density multiplied by a constant will dominate the original mixed-exponential density. In the next section this method was used to obtain the Monte Carlo results. 

(iv) Since $\sigma$ is strictly positive, Assumption~\ref{as:smooth} is satisfied for the MEJD process $X$, and $X_{\Gamma_{n,q}}$, $n\in\mbb N, q>0$, has a density 
by Lemma~\ref{lem:density}.  
\end{rems}

From the definition of the MEJD process $X$ it is straightforward to verify that the characteristic exponent 
$\Psi(s)=-\log\E[\te{\mbf isX_1}]$ 
is a rational function of the form
$$
\Psi(s) = - \mbf i\mu s + \frac{\sigma^2 s^2}{2} - \lambda \left( \sum_{i=1}^{m^+}p_i^+\frac{\alpha_i^+}{\alpha_i^+-\mbf i s}  + \sum_{j=1}^{m^-}p_j^-\frac{\alpha_j^-}{\alpha_j^-+\mbf i s}  - 1 \right), \quad\quad s\in\mbb R.
$$
The distributions of $X$, the running supremum $\ovl X$ and the running infimum $\unl X$ 
at the random time $\Gamma_{1,q}$ and also the functions $D_{1,q}$ and $\Omega_{1,q}$ can be expressed, as we shall see below, in terms of the roots 
$\{\rho_k^+, k=1, \ldots, m^++1\}$ and $\{\rho_k^-, k=1, \ldots, m^-+1\}$ with positive and negative real parts 
of the Cram\'er-Lundberg equation
\begin{equation}
\label{eq:CramerLundberg}
q +  \Psi(-\mbf i s)= 0,\quad q>0.
\end{equation}
For the MEJD $X$ the Wiener-Hopf factors $\Psi^+_q$ and $\Psi^-_q$ can be identified explicitly.
It is well-known that  $\Psi_q^+(\theta)$ 
and $\Psi_q^-(\theta)$ have neither zeros nor poles on the half-planes $\{\Im(z)>0\}$ and $\{\Im(z)<0\}$
respectively, as a consequence of the fact that $\Psi_q^+$ 
and $\Psi_q^-$ are the characteristic functions of infinitely divisible distributions 
supported on the positive and negative half-lines respectively (see \cite[Ch. 9]{Sato1999}). 
In particular, using that $\Psi^+_q(\theta)$ 
and $\Psi^-_q(\theta)$ satisfy  $q/(q+\Psi(\theta)) =\Psi^+_q(\theta)\Psi^-_q(\theta)$ 
for $\theta\in\mbb R$, the Wiener-Hopf factors of a mixed-exponential jump-diffusion 
can be identified as certain rational functions (see \cite{LewisMoredcki}):

\begin{Lemma}Let $q>0$ be given. The functions $\Psi^+_q$ and $\Psi^-_q$
are given explicitly by
\begin{eqnarray}
\label{eq:Psi+}
\Psi^+_q(s) & :=  & \prod_{i=1}^{m^+}\left(1-\mbf i s/\alpha_i^+\right) 
\prod_{i=1}^{m^++1}\left(1-\mbf i s/\rho_i^+(q)\right)^{-1},\\
\label{eq:Psi-}
\Psi^-_q(s) & := & \prod_{j=1}^{m^-}\left(1+\mbf i s/\alpha_j^-\right) 
\prod_{j=1}^{m^-+1}\left(1-\mbf i s/\rho_j^-(q)\right)^{-1}.
\end{eqnarray}
\end{Lemma}

The fact that the Wiener-Hopf factors $\Psi^+_q$ and $\Psi^-_q$ are rational functions
implies that, when the roots of the Cram\'{e}r-Lundberg equation are distinct, 
the running supremum $\ovl X_{\Gamma_{1,q}}$ and infimum $\unl X_{\Gamma_{1,q}}$
of $X$ at $\Gamma_{1,q}$ , where $\ovl X_t := \sup_{s\leq t}X_s$ and $\unl X_t := \inf_{s\leq t}X_s$ denote 
the running supremum and infimum of $X$ at $t\in\mbb R_+$, also follow
mixed-exponential distributions. 

\begin{Lemma}\label{lem:XGq}Let $q>0$ be given and suppose that the roots of \eqref{eq:CramerLundberg}
are distinct.
The random variables $\ovl X_{\Gamma_{1,q}}$, $-\unl X_{\Gamma_{1,q}}$ and $X_{\Gamma_{1,q}}$
have mixed-exponential distributions with densities 
$\ovl u_{1,q}$, $\unl u_{1,q}$ and $u_{1,q}$
given by
\begin{eqnarray}
&& \ovl u_{1,q}(x) = 
\sum_{i=1}^{m^++1}A_i^+(q)\rho_i^+(q)\te{-\rho_i^+(q)x}, 
\q \unl u_{1,q}(x) = 
\sum_{j=1}^{m^-+1}A_j^-(q)(-\rho_j^-(q))\te{\rho_j^-(q)x},\q x>0,\\
&& u_{1,q}(x) =  \sum_{i=1}^{m^++1} B_i(q)\te{-\rho_i^+(q)x} I_{(0,\infty)}(x)+ 
\sum_{j=1}^{m^-+1}C_j(q) \te{-\rho_j^-(q)x}I_{(-\infty,0)}(x),\q x\in\mbb R,
\label{eq:PotentialDensity}
\end{eqnarray}
with, for $i=1, \ldots, m^++1$ and $j=1, \ldots, m^-+1$,
\begin{eqnarray}
&& A_i^+(q) := \frac{\prod_{k=1}^{m^+}(1-\rho_i^+(q)/\alpha_k^+)}{\prod_{k\ne i}(1-\rho_i^+(q)/\rho_k^+(q))},
\qquad
A_j^-(q) := \frac{\prod_{k=1}^{m^-}(1+\rho_j^-(q)/\alpha_k^-)}{\prod_{k\ne j}(1-\rho_j^-(q)/\rho_k^-(q))},\\
&& B_i(q) := A_i^+(q)\Psi^-_q(\rho_i^+(q))\rho^+_i(q), 
\qquad C_j(q) := A_j^-(q)\Psi^+_q(\rho_j^-(q))(-\rho_j^-(q)),\label{eq:BCij}
\end{eqnarray}
where we define $A_k^\pm\equiv1$ in the case 
$m^\pm=0$ (i.e. if there are no positive/negative jumps). 
\end{Lemma}
\begin{proof}
It is straightforward to verify that 
the coefficients of the function $(1-\mathbf i s/\rho_i^+(q))^{-1}$ in the partial-fraction decompositions of 
the functions $q/(q+\Psi(s))$ and $\Psi^+_q(s)$ are given by $C_i(q)$ and $A_i^+(q)$, respectively, while the coefficients of 
the function $(1-\mathbf i s/\rho_j^-(q))^{-1}$ in the partial-fraction decompositions of the functions $q/(q+\Psi(s))$ and $\Psi^-_q(s)$ 
are given by $B_j(q)$ and $A_j^-(q)$ respectively. Subsequently inverting the 
Fourier transforms $(1-\mathbf i s/\rho_i^+(q))^{-1}$ and $(1-\mathbf i s/\rho_j^-(q))^{-1}$ 
yields the stated expressions for the densities of $\ovl X_{\Gamma_{1,q}}$, $-\unl X_{\Gamma_{1,q}}$ and $X_{\Gamma_{1,q}}$.
\end{proof}

The functions $\Omega_{n,q}$ and $D_{n,q}$ and the density $u_{n,q}$ can be explicitly identified by combining the forms of the functions $\Omega_{1,q}$ and $D_{1,q}$ (identified below) with the recursive relations in \eqref{eq:Drec} and \eqref{eq:Orec}.
From the form of these recursive relations it follows that 
the functions $\Omega_{n,q}$, $D_{n,q}$ and $u_{n,q}$ can be expressed as 
linear combinations of exponentials  
with the weights given by certain polynomials---the explicit expressions are given in the following result.

Consider the  polynomials 
$\WT P^\pm_{k,i,n}$, $\WT P^\pm_{i,j,k,n}$
and real numbers $\WT c^\pm_{i,j,n}$
defined by
\begin{eqnarray*}
&& \int_0^x P^+_{k,n}(y) \te{-\rho^+_k y - \rho_i^+(x-y)}\td y = \te{-\rho_k^+ x} \WT P^+_{k,i,n}(x) - 
\te{-\rho_i^+ x} \WT c^+_{k,i,n},\\ 
&& \int_x^0 P^-_{k,n}(y) \te{-\rho_k^-y -\rho_i^-(x-y)} \td y = 
\te{-\rho_k^- x} \WT P^-_{k,i,n}(x) - 
\te{-\rho_i^- x} \WT c^-_{k,i,n},\\
&& \int_0^x \te{\rho^+_i (z-x)} P_{i,j,n}(x-z,y-z)u_1(z)\td z = 
\sum_{k=1}^{m^++1}\WT P^+_{i,j,k,n}(x,y) \te{-\rho_k^+x},\\ 
&& \int_0^x \te{-\rho^-_j (x-z)} u_n(z)\td z = 
\sum_{k=1}^{m^-+1}\WT P^-_{i,j,k,n}(x) \te{-\rho_k^-x},
\end{eqnarray*}
where we denoted $\rho_h^+=\rho_h^+(q)$ and $\rho_h^-=\rho_h^-(q)$, and 
$P^+_{k,n}$ and $P^-_{k,n}$ are the polynomials to be defined shortly.
The fact that there exist polynomials and constants satisfying the above relations
follows by repeated integration by parts.  
By induction the following expressions for the functions $u_{n,q}$, $D_{n,q}$ and $\Omega_{n,q}$ can be derived:

\begin{Prop}\label{prop:DOU}
For any $n\in\mbb N\cup\{0\}$ we have 
\begin{eqnarray*}
u_{n+1,q}(x) &=& \sum_{k=1}^{m^++1} P^+_{k,n+1}(x) \te{-\rho^+_kx}I_{(0,\infty)}(x) + 
\sum_{k=1}^{m^-+1} P^-_{k,n+1}(x) \te{-\rho^-_kx}I_{(-\infty,0)}(x), \q x\in\mbb R,\\
D_{n+1,q}(x,y) &=& u_{n+1,q}(y) - 
\sum_{i=1}^{m^++1}\sum_{j=1}^{m^-+1} P_{i,j,n+1}(x,y)\te{-\rho_j^-(y-x) - \rho_i^+ x}, 
\q\text{$x\in\mbb R_+, x\ge y$}, \\
\Omega_{n+1,q}(x,y) &=& q^{-(n+1)}\cdot \sum_{k=1}^{n+1} u_{n+2-k,q}(x) u_{k,q}(y-x), \qquad \text{$x,y\in\mbb R$},
\end{eqnarray*}
with as before $\rho_j^-=\rho_j^-(q)$ and $\rho_i^+=\rho_i^+(q)$, and with
 $P^+_{k,1}\equiv B_k(q)$, $P^-_{k,1} \equiv C_k(q)$ and
$P_{i,j,1} \equiv \frac{E_{ij}(q)}{\rho_j^- - \rho_i^+} := \frac{A_i^+(q)A_j^-(q)\rho^+_i(q)\rho_j^-(q)}{\rho_j^- - \rho_i^+}$, 
and where $P^\pm_{k,n+1}$ and $P_{i,j,n+1}$ are polynomials 
and $c^\pm_{k,i,n}$ are real numbers
that are defined recursively for $n\in\mbb N$, as follows:
\begin{eqnarray*}
&& P^+_{k,n+1}(x) = \sum_{r=1}^{m^-+1} \le(C_r(q)\int_0^\infty \te{(\rho_r^--\rho_k^+)z}P^+_{k,n}(x+z)\td z 
+  B_k(q)c^-_{k,r,n}\ri)
+ \sum_{r=1}^{m^++1} B_r(q)\le(\WT P^+_{k,r,n}(x) - \WT c^+_{r,k,n} \ri),\\
&& P^-_{k,n+1}(x) = \sum_{r=1}^{m^++1} \le(B_r(q) \int_{-\infty}^0 \te{(\rho_r^+-\rho_k^-)z}P^-_{k,n}(x+z)\td z + C_k(q)c^+_{k,r,n}\ri)
+ \sum_{r=1}^{m^-+1} C_r(q)\le(\WT P^-_{k,r,n}(x) - \WT c^-_{r,k,n}\ri),\\
&& P_{i,j,n+1}(x,y) = \int_{-\infty}^0 P_{i,j,n}(x-z,y-z)\te{\rho^+_iz}u_1(z)\td z 
+ \sum_{k=1}^{m^++1}\WT P^+_{k,j,i,n}(x,y) - \sum_{k=1}^{m^-+1}\WT P^-_{i,k,j,n}(y-x)\\
&& +\ B_i(q) \int_0^\infty P_{j,n}^-(y-x-z) \te{(\rho_j^- - \rho_i^+)z}\td z
  + \frac{E_{i,j}(q)}{\rho_j^- - \rho_i^+} \int_{0}^{\infty}u_{n,q}(z) \te{\rho_j^- z}\td z\\
&& - \sum_{k=1}^{m^++1}\sum_{l=1}^{m^-+1} \frac{E_{k,l}(q)}{\rho_l^- - \rho_k^+}\int_{-\infty}^0 P_{i,j,n}(-z,y-x-z)\te{\rho_i^+z - \rho_l^-z}\td z,\\
&& c^-_{k,r,n} = \int_{-\infty}^0 \te{(\rho_k^+-\rho_r^-)z}P^-_{r,n}(z)\td z, \q 
c^+_{k,r,n} = \int_0^{\infty} \te{(\rho_k^--\rho_r^+)z}P^+_{r,n}(z)\td z.
\end{eqnarray*}
\end{Prop}

\begin{proof}
By combining the identity 
$\PP[\overline{X}_{\Gamma_{1,q}} \in\td x, x - X_{\Gamma_{1,q}} \in\td z] 
= \PP[\overline{X}_{\Gamma_{1,q}} \in\td x]\PP[ -\unl X_{\Gamma_{1,q}} \in\td z]$, $x,z\in\mbb R_+$,
(which follows from the Wiener-Hopf factorisation of $X$)
with Lemma~\ref{lem:XGq} and performing a one-dimensional integration, we get the expression for the function $D_{1,q}$. 
The Markov property and stationarity of increments yields $\Omega_{1,q}(x,y)= q^{-1}  u_{1,q}(y-x)\>u_{1,q}(x)$, whence 
we have the form of the function $\Omega_{1,q}$ by inserting the expression~\eqref{eq:PotentialDensity} for $u^q$ .
The expressions for $u_{n+1,q}$, $D_{n+1,q}$ and $\Omega_{n+1,q}$ follow by induction with respect to $n$, utilising (i) the fact that $u_{n+1,q}$ is equal to the convolution of $u_{n,q}$ and $u_{1,q}$, as a 
consequence of the independence and stationarity of the increments of $X$, 
(ii) the form of $D_{1,q}$ 
and the recursive relation in \eqref{eq:Drec}, 
and (iii) the form of $\Omega_{1,q}$  and 
the recursive relation in \eqref{eq:Orec}.
\end{proof}

\section{Convergence and error-estimates}\label{sec:conv}

The randomisation method consists in approximating the value $f(t)$ of a function $f$ at time $t>0$ 
  by the expectation $\E[f(\Gamma_{n,n/t})]$ of $f$ evaluated at a random time $\Gamma_{n, n/t}$ that follows a Gamma distribution 
	with expectation $\E[\Gamma_{n,n/t}] = t$ and variance $\E[(\Gamma_{n,n/t}-t)^2] = t^2/n$. 
Since the random variables $\Gamma_{n, n/t}$  converges in distribution to $t$ as $n$ tends to infinity, the error $\E[f(\Gamma_{n,n/t})]-f(t)$ converges to zero 
for any bounded and continuous function $f$.  The error can be expanded in terms of powers of $1/n$ provided that $f$ is sufficiently smooth, as shown in the following result:
\begin{Thm}\label{lem:conv} Let $k$ be a given non-negative integer and consider 
$f\in C^{2k+2}(\mbb R_+)$. There exist functions 
$b_1, \ldots, b_{k+1}:\mbb R_+\to\mbb R$ such that we have, for any $t\in\mbb R_+$,
\begin{equation}\label{eq:fGnn}
n^{k+1}\le[\E[f(\Gamma_{n,n/t})] - f(t) - \sum_{m=1}^k b_m(t) \le(\frac{1}{n}\ri)^m\ri] = 
b_{k+1}(t) + o(1)\q\text{as $n\to\infty$}.
\end{equation}
In particular, denoting by $f^{(m)}$ the $m$th derivative of $f$, we have 
\begin{eqnarray*}
&& b_1(t) = \frac{t^2}{2} f^{(2)}(t),\q b_2(t) = \frac{t^4}{8} f^{(4)}(t) + \frac{t^3}{3} f^{(3)}(t),\q
b_3(t) = \frac{t^6}{48}f^{(6)}(t) + \frac{t^5}{6}f^{(5)}(t) + \frac{t^4}{4}f^{(4)}(t),\\
&& b_4(t) = \frac{t^8}{384} f^{(8)}(t) + \frac{t^7}{24}f^{(7)}(t) + \frac{13t^6}{72}f^{(6)}(t) + \frac{t^5}{5} f^{(5)}(t).
\end{eqnarray*}
\end{Thm}
\begin{rems}\rm
(i) Theorem~\ref{lem:conv} implies that for $f\in C^2(\mbb R_+)$ the error of the approximation of $f(t)$ by 
$\E[f(\Gamma_{n,n/t})]$ decays linearly, that is, 
$\E[f(\Gamma_{n,n/t})] - f(t) = \frac{b_1(t)}{n} + o(\frac{1}{n})$ as $n$ tends to infinity.

(ii) Theorem~\ref{lem:conv} also provides a justification of the use of the Richardson extrapolation to increase the speed of convergence if the function $f$ is sufficiently smooth. 
Since the error of the approximation is given in terms of positive integer powers of $1/n$,
 the Richardson extrapolation that utilises the first $N$ values 
$\E[f(\Gamma_{1,1/t})]$, $\ldots$, $\E[f(\Gamma_{N,N/t})]$ is explicitly given by
 \begin{eqnarray}
P_{1:N} &=& \sum_{k=1}^N \frac{(-1)^{N-k} k^N}{k! (N-k)!}\E[f(\Gamma_{k,k/t})], \label{eq:MarchukEnhancement}
\end{eqnarray}
(see \cite[\S1.3]{Marchuk1983} for a derivation of this formula). Note in particular 
that in order to deploy the extrapolation \eqref{eq:MarchukEnhancement} it suffices to know the existence of functions $b_m$ 
such that \eqref{eq:fGnn} holds and it is not required to find their explicit form. 
In the case $f\in C^{2k+2}(\mbb R_+)$, $k<N$ Theorem~\ref{lem:conv} implies that
the error $P_{1:N} - f(t)$ of the interpolation $P_{1:N}$ 
is $o(N^{-k-1})$. In particular, if $f$ is $C^\infty$ 
then the error $P_{1:N} - f(t)$  is $O(N^{-k-1})$ for every $k$, as $N$ 
tends to infinity. Refer to \cite{Sidi2003} for background on the theory of extra- and interpolation.
\end{rems}

\begin{proof}[Proof of Theorem \ref{lem:conv}]
While we expect this result to be known in the literature, we have not been able to find a reference and provide a brief 
proof. Taylor's theorem and the fact that $f\in C^{2k+2}$ imply
$$
f(s) - f(t) = \sum_{m=1}^{2k+1}\frac{(s-t)^m}{m!} f^{(m)}(t) + R(s,t)
$$
where the remainder term is given by $R(s,t) = \frac{(s-t)^{2k+2}}{(2k+2)!}f^{(2k+2)}(\xi)$ for some $\xi$ between $s$ and $t$. 
Replacing $s$ by the independent Gamma random variable $\Gamma_{n,n/t}$ we get
$$
\E[f(\Gamma_{n,n/t}) - f(t)] = \sum_{m=2}^{2k+1} \frac{a_{m,n}}{m!} f^{(m)}(t) + \E[R(\Gamma_{n,n/t},t)]
$$
with $a_{m,n}=\E[(\Gamma_{n,n/t}-t)^m]$, where we have $a_{1,n}=0$ as the expectation $\E[\Gamma_{n,n/t}]$ is equal to $t$. 
The numbers $a_{m,n}$ are equal to $a_{m,n} = \le.\frac{\td^m}{\td u^m}\ri|_{u=0} M(u)$ where 
$M$ denotes the moment-generating function of the random variable $\Gamma_{n,n/t}-t$ which is given by
$$
M(u) = \le(1 - \frac{ut}{n}\ri)^{-n} \exp\{- ut \}, \qquad u \leq \frac{n}{t}.
$$
In particular, it follows from the form of $M$ that the $a_{m,n}$ are linear combinations 
of positive integer powers of $1/n$. Reordering of terms and straightforward manipulations result in the identity in 
\eqref{eq:fGnn}. 
\end{proof}

We next turn to the problem of approximation of the distribution of the supremum and
the expected occupation time of the set $(-\infty,x]$ 
of the {\it L\'{e}vy bridge process} $X^{(0,0)\to (t,y)}$ from $(0,0)$ to $(t,y)$ 
(its definition is recalled in Appendix~\ref{sec:BridgeSampling}): 
\begin{eqnarray}\label{eq:md}
&&\vec{d}_{t}(x,y):=\PP\le(\ovl X^{(0,0)\to(t,y)}\leq x\ri),\q
\vec{\w}_{t}(x,y):=\E\le[\int_0^t I_{\le\{X_u^{(0,0)\to(t,y)}\leq x\ri\}}
\td u\ri],\q\text{with}\\
&& \ovl X^{(0,0)\to(t,y)}:= \sup_{u\in[0, t]} X^{(0,0)\to(t,y)}_u.\nonumber
\end{eqnarray}
By spatial and temporal homogeneity of $X$, 
the corresponding quantities in the case of a general 
starting point $(s,z)$ are given in terms 
of $\vec d$ and 
$\vec\w$ by $\vec{d}_{t-s}(x-z,y-z)$ and $\vec\w_{t-s}(x-z,y-z)$.
The approximations of $\vec d$ and 
$\vec\w$ are given 
in terms of the randomised bridge process $X^{(0,0)\to(\Gamma_{n,q},y)}$ (see Appendix~\ref{sec:BridgeSampling})
as follows:
\begin{eqnarray*}
&& \vec D^{(n)}_q(x,y) := \PP\le(\ovl X^{(0,0)\to(\Gamma_{n,q},y)}
\leq x\ri), \q \vec\Omega^{(n)}_q(x,y) 
:= \E\le[\int_0^{\Gamma_{n,q}} I_{\le\{X_u^{(0,0)\to(\Gamma_{n,q},y)}
\leq x\ri\}}\td u\ri].
\end{eqnarray*}
We derive next error estimates for these randomised bridge approximations.

\begin{Cor}\label{cor:error}Let $x$, $y\in\mbb R$ and $t>0$. For some constants $C^d$ and $C^\w$ 
we have, for all positive integers $n$,
\begin{equation}\label{eq:est}
\le|\vec D^{(n)}_{n/t}(x,y) - \vec{d}_t(x,y)\ri| \leq \frac{C^d}{n}, \q 
\le|\vec \Omega^{(n)}_{n/t}(x,y) - \vec{\w}_t(x,y)\ri| \leq \frac{C^\w}{n}.
\end{equation}
\end{Cor}
\begin{proof} Since
the distribution of $X_t$ has a continuous density $p_t(y)$
and $s\mapsto \vec d_s(x,y)$, 
$s\mapsto \vec\w_s(x,y)$ and $s\mapsto p_s(y)$ are $C^2$ at $s=t$ with $p_t(y)>0$,
the estimates in \eqref{eq:est} 
follow by applying Theorem~\ref{lem:conv} to the functions $t\mapsto\vec{d}_t(x,y)p_t(y)$, $t\mapsto\vec{\w}_t(x,y)p_t(y)$ and $t\mapsto p_t(y)$.
\end{proof}

\section{Numerical illustration: first-passage time probabilities 
and occupation times}
\label{sec:NumericalResults}
To provide a numerical illustration of the randomisation method,
we implemented the recursive formulas (given in Proposition~\ref{prop:DOU}) 
to approximate the following expectations of path-dependent functionals:
\begin{eqnarray*}
\PP\le(\sup_{u\in[0, t]} X^{(0,x)\to(t,y)}_u \leq z\ri),
\q \E\le[\int_0^t I_{\le\{X_u^{(0,x)\to(t,y)}\in (a,b)\ri\}}
\td u\ri]\q
\begin{array}{l}
x=1, y=1.1, z=1.2,\\ 
t=1, a = 1.05, b = 1.25,
\end{array}
\end{eqnarray*}
for the case\footnote{See \cite[Chapter 3]{Johannes-thesis} for additional numerical examples.}
that the underlying 
L\'{e}vy process $X$ is equal to  a 
HEJD process with typical parameters, which are detailed 
in Table~\ref{tab:ModelParameters}.  The outcomes 
are reported in Table~\ref{tab:ConvergenceOneSided}
and Figure~\ref{fig:Convergence}. In Table~\ref{tab:ConvergenceOneSided} 
the values are listed of the first-passage probabilities and the expected occupation times of the randomised L\'{e}vy bridges corresponding to a 
$\Gamma(n,n)$-randomisation of the fixed time $T=1$ for a number of values of $n$.  
We also reported the results obtained by applying a Richardson extrapolation $P_{1:n}$ of order $n$, using the first $n$ outcomes 
(defined in \eqref{eq:MarchukEnhancement}).  
The logarithms of the corresponding absolute errors are plotted 
in Figure~\ref{fig:Convergence}. The errors were computed 
with respect to the value $P_{1:11}$ that was obtained 
after Richardson's extrapolation with $n=11$ stages. 

\begin{table}[thbp!]
\caption[Chosen model parameters]{\small The model parameters used throughout the paper. 
The parameters for the Kou model are taken from \cite{Kou2004}, the ones for the HEJD model from~\cite{Jeannin2010}, and the ones for the MEJD model from \cite{Cai2011} (which for the latter two models
have been re-expressed using our notation).~\label{tab:ModelParameters}}
\begin{center}
\footnotesize{
\begin{tabular}{lllll}
\hline
& KOU & HEJD  & MEJD  \\
\hline
  $\sigma$ & 0.2 & $\sqrt{0.042}$ & 0.2  \\
  $\lambda$  & 3.0 & 11.5 & 1.0 \\
  $\alpha^+$  & 50 & (5, 10, 15, 25, 30, 60, 80) & (213.0215, 236.0406, 237.1139, 939.7441, 939.8021) \\
  $\alpha^-$  & 25 & (5, 10, 15, 25, 30, 60, 80) & (213.0215, 236.0406, 237.1139, 939.7441, 939.8021) \\
  $p^+$  & 0.3 & $(0.05, 0.05, 0.1, 0.6, 1.2, 1.9, 6.1)*0.51/\lambda$ & (4.36515, 1.0833, -5, 0.0311, 0.02045) \\
  $p^-$  & 0.7 & $(0.5, 0.3, 1.1, 0.8, 1, 4, 2.3)*0.64/\lambda$ & (4.36515, 1.0833, -5, 0.0311, 0.02045) \\
\hline
\hline 
\end{tabular}
}
\end{center}
\end{table}
\begin{figure}[tp!]
\centering
\noindent\mbox{\subfigure{\includegraphics[scale=0.4]{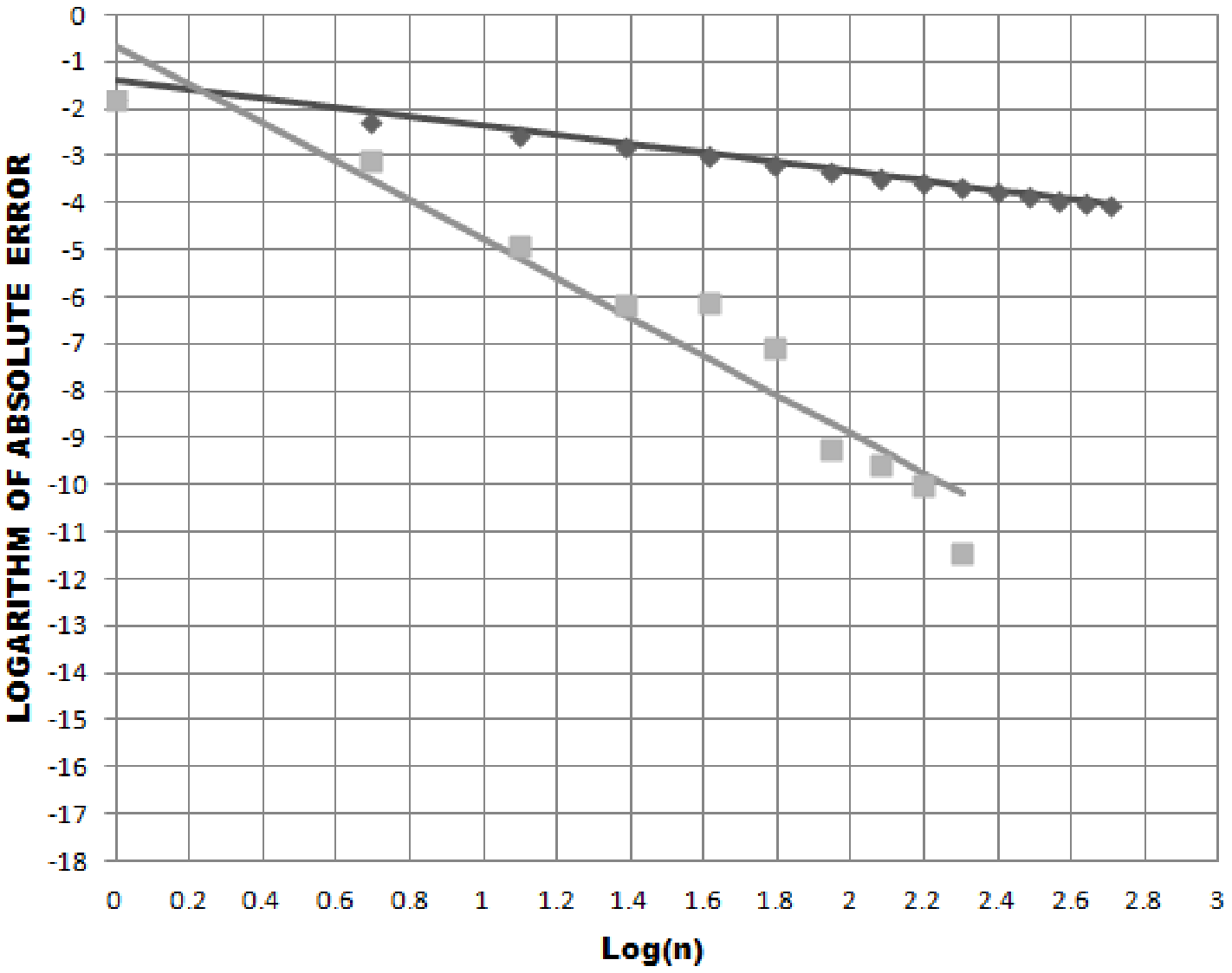}}
\subfigure{\includegraphics[scale=0.4]{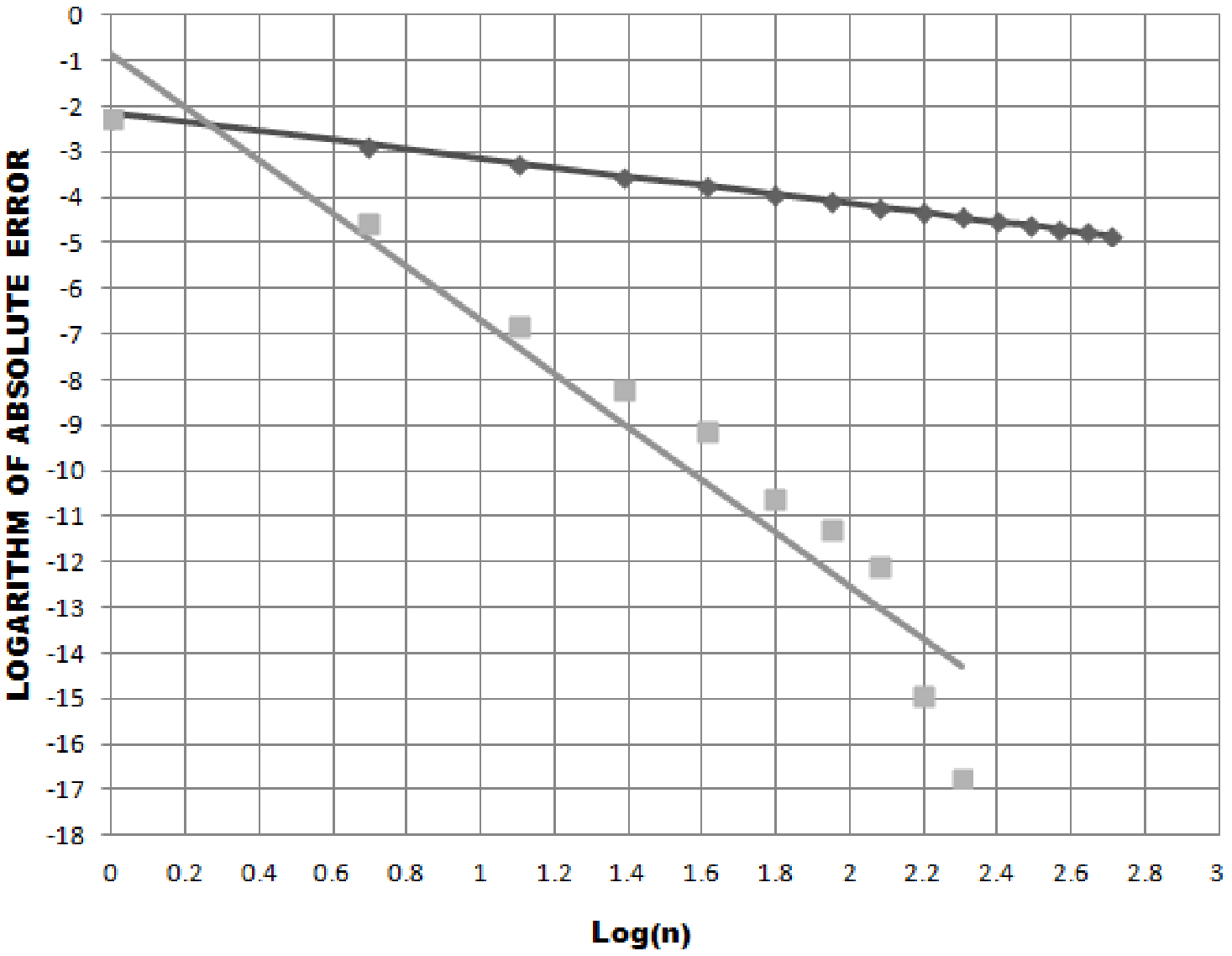}}}\\
\subfigure{\includegraphics[scale=0.7]{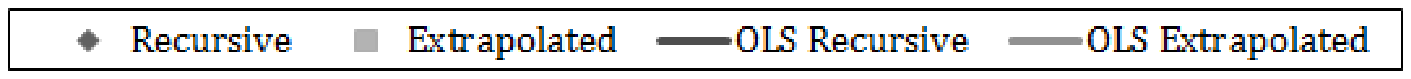}}
\caption{\footnotesize The logarithms of the absolute errors 
of the outcomes generated by the recursive algorithm for (a) the one-sided first-passage probabilities 
and (b) the expected occupation time under the HEJD processes as a function of $n$, where $n$ is the number of steps in the recursions. 
In each sub-figure the errors of the recursive values and the Richardson extrapolated values are displayed. Also ordinary least square estimations of either series of results are plotted (in the case of the un-extrapolated values the OLS line was estimated using the last six values only). The slopes of the dark lines in sub-figures (a) and (b) are 
given by $-0.98$ and $-0.99$, respectively.  
The starting point of the bridge is $1.0$, the end point is $1.1$ and the barrier level is $1.2$ and the range is $(1.05,1.25)$. In all cases the L\'{e}vy bridge process is assumed to start at time $0$ and to end at time $1$. The model parameters that were used are given in Table~\ref{tab:ModelParameters}.}
\label{fig:Convergence}
\end{figure}
{\footnotesize
\begin{table}[tp!]
\caption{\footnotesize Approximations of one-sided first-passage time (FPT) probabilities  and expected occupation times obtained recursively ($P_n$) and with Richardson extrapolation ($P_{1:n}$) for the HEJD model as a function of $n$, where $n$ is the number of recursions. 
The starting point of the bridge is assumed to be $1.0$, the end point is $1.1$, the barrier level is $1.2$ and the range is  $(1.05,1.25)$.
In all cases the L\'{e}vy bridge is assumed to start at time $0$ and to 
end at time $1$. The model parameters r
that were used are given in Table~\ref{tab:ModelParameters}. 
}
\begin{center}
\begin{tabular}{rr|c|rr}
\hline
\multicolumn{3}{l}{FPT probability}&  \multicolumn{2}{r}{Expected occupation time}\\
\hline
      $ P_n $ HEJD & $ P_{1:n} $ HEJD & $n$
&  $ P_n $ HEJD & $ P_{1:n} $ HEJD \\
\hline
    0.3006853 & 0.3006853 & 1      & 0.3680801 & 0.3680801 \\
    0.3617512 & 0.4228170 & 2      & 0.4142655 & 0.4604509 \\
    0.3911554 & 0.4635372 & 3      & 0.4322124 & 0.4719338 \\
    0.4084846 & 0.4734619 & 4      & 0.4415893 & 0.4711338 \\
    0.4198448 & 0.4735378 & 5      & 0.4473202 & 0.4707490 \\
    0.4278257 & 0.4720958 & 6      & 0.4511786 & 0.4708328 \\
    0.4337174 & 0.4713210 & 7      & 0.4539517 & 0.4708704 \\
    0.4382332 & 0.4711443 & 8      & 0.4560403 & 0.4708630 \\
    0.4417979 & 0.4711707 & 9      & 0.4576699 & 0.4708578 \\
    0.4446794 & 0.4712065 & 10     & 0.4589767 & 0.4708575 \\
    0.4470546 & 0.4712177 & 11     & 0.4600480 & 0.4708575 \\

\hline
\hline 
\end{tabular}  
\label{tab:ConvergenceOneSided}
\end{center}
\end{table}
}
Empirically we observe that the rate of decay of the error of the un-extrapolated outcomes to be (approximately) linear for both different functionals, 
in line with the theoretical error bound 
given in Corollary~\ref{cor:error}: indeed, the ordinary least squares (OLS) regression lines (dark grey) in the log-log plots had slopes equal to $-0.94$ ($-0.98$) and $-0.98$ ($-0.99$) in the case of the first-passage probabilities (and expected occupation times) of the L\'{e}vy 
bridges corresponding to the HEJD model. 
Moreover, in line with the theoretical error estimates 
given in Theorem~\ref{lem:conv}, we observe that the application of the Richardson extrapolation leads to a significantly faster decay of the error.  By comparing the error plots of the expectations of the two path-dependent functionals 
we note that the logarithmic errors for the expected occupation times (for a given $n$) are consistently and significantly the smaller of the two, suggesting that the randomisation method converges faster in this case. This feature is likely to be related to the higher degree of smoothness  in the case of the expected occupation time. Finally, we mention that we computed the 
roots the Cram\'{e}r-Lundberg equation featuring in the solutions $D_{n,q}$ and $\Omega_{n,q}$ by deploying the Newton-Raphson method.\footnote{We investigated the round-off error resulting from the computation of the roots based on single precision arithmetic, and found that in that case the computed roots were accurate up to an error of $1.0e^{-11} $.}$\!\phantom{|}^,$\footnote{In order to efficiently approximate the first-passage time probability and the expected occupation time of the L\'{e}vy bridge process, one could combine the procedure described in 
this section with interpolation: 
One would then compute these quantities for a grid of points and construct subsequently functions on the real line $\mbb R$ by using 
(linear) interpolation.}

\section{Illustration: Option valuation using the bridge sampling method}\label{sec:option}

By way of illustration we next present the numerical results that were 
obtained by valuing an up-and-in barrier option and a range note under a number of models by using a Markov bridge algorithm described in Table~\ref{table:bridge} below 
(the recursive method for approximation of first-passage time probabilities and expected occupation times from Section~\ref{sec:NumericalResults} is applied).

We assume that the stock price process $S = \{S_t, t\in\mbb R_+\}$ evolves according to a Bates-type 
stochastic volatility model with mixed-exponential jumps.
The process $S$  is thus specified by 
the exponential model
$$S_t = \exp\{Y_t\}, \q t\in\mbb R_+,$$
where the log-price process $Y=\{Y_t, t\in\mbb R_+\}$ satisfies the stochastic differential equation 
\begin{eqnarray}
\td Y_t &=& \left( \mu-\frac{Z_t}{2} \right) \td t + \sqrt{|Z_t|} \td B_t +  \td J_t,\q Y_0 = x, \label{Eqn:Bates} \\
\td Z_t &=& \kappa ( \delta - Z_t) \td t + \xi \sqrt{|Z_t|}\td W_t, \ \ \ \ \  t \in\mbb R_+,\q  Z_0 = v, 
\end{eqnarray}
where $x$ and $v$ are strictly positive, $(B,W)$ is a two-dimensional Brownian motion with correlation-parameter $\rho$ and
$ J_t $ is an independent compound Poisson process with intensity $ \lambda $ and jump-sizes 
distributed according to a mixed-exponential distribution $F$ with mean $m$. The parameters $\kappa$, $\delta$, and $\xi$
of the model are positive and represent the speed of mean-reversion of the volatility, the long term volatility level and the volatility of volatility parameter. The parameter $\mu$ is set equal to 
$\mu=r-q - \lambda m$ which ensures that the moment condition 
$\E[\exp\{Y_t\}] = \exp\{(r-q)t+ Y_0\}$ is satisfied for all non-negative $t$, where the constants $r$ and $q$ 
are non-negative constants representing the risk-free rate of return and the dividend yield. 
Under this moment condition it holds that the process $\{\te{-(r-q)t}S_t, t\in\mbb R_+\}$ is a martingale.
Note that choosing $\kappa$ and $\xi$ equal to zero yields the mixed-exponential jump-diffusion process.  

By way of example we consider an up-and-in call (UIC) option and a range note (RN).  By arbitrage pricing theory, the UIC option and the 
RN have values at time $0$ given by
\begin{eqnarray*}
UIC(K,H) =  \E  \left[\te{-rT} (S_T-K)^+ I_{\{ \sup_{0 \leq t \leq T} S_t > H  \}} \right],\ \ 
RN(a_1,a_2) =  \E \left[\te{-rT}\cdot \frac{C}{T}\int_0^T I_{ \{ a_1 \leq S_u \leq a_2 \} } \td u \right], 
 \end{eqnarray*}
where $K$ is the strike price, $H$ is the barrier level, $C$ is 
the nominal, and $a_1$ and $a_2$ are the lower and upper 
bound of the range respectively. 

\subsection{Markov Bridge sampling method}\label{sec:bridge}

The first step is to approximate the log-price process $Y$ by a process that has piecewise constant drift and volatility 
deploying the Euler-Maruyama approximation of the process $(Y,Z)$ 
on the equidistant partition $\mbb T_N$ which can be expressed as 
\begin{eqnarray}
\label{eq:Yp1}
&& Y'_{\tau_{n+1}} = Y'_{\tau_{n}} + \le(\mu - \frac{Z'_{\tau_n}}{2}\ri)\Delta_n + 
 \sqrt{|Z'_{\tau_n}|}\Delta W_n + \Delta J_n,\q Y'_0 = x,\\
\label{eq:Zp1}
&& Z'_{\tau_{n+1}} = Z'_{\tau_{n}} + \kappa(\delta - Z'_{\tau_n})\Delta_n + 
 \xi\sqrt{|Z'_{\tau_n}|}\Delta B_n,\q Z_0' = v,
\end{eqnarray}
for $n\in\mbb N\backslash\{0\}$,
with $\Delta W_n = W_{\tau_{n+1}} - W_{\tau_{n}}$, 
$\Delta B_n = B_{\tau_{n+1}} - B_{\tau_{n}}$,
$\Delta J_n = J_{\tau_{n+1}} - J_{\tau_{n}}$, and $\Delta_n = \tau_{n+1}-\tau_n=T/N$. 
See \cite{HM,KlNe} for results 
on strong and weak-convergence of this scheme.
The Markov bridge-sampling method is based on the
continuous-time Euler-Maruyama approximation $Y'$ leaving the (piecewise constant) approximation 
$(Z'_{\tau_n})_{n\in\mbb N}$ for $Z$  given in \eqref{eq:Zp1} unchanged. 
We arrive at the approximation
\begin{eqnarray}\label{eq:Yp2}
&& Y'_{t} = Y'_{\tau_{n}} + \le(\mu - \frac{Z_{\tau_n}}{2}\ri)(t-\tau_n)  + 
 \sqrt{|Z'_{\tau_n}|}(W_t - W_{\tau_n}) +  (J_t - J_{\tau_n}),
\\
&& Z'_t = Z'_{\tau_n},
\label{eq:Zp2}
 \end{eqnarray}
for $t\in[\tau_n, \tau_{n+1}]$.  Observe that with this choice of interpolation it holds that,
conditional on the values of the random variable $Z'_{\tau_n}$, 
the process $\{Y'_{t-\tau_n}, t\in [\tau_n, \tau_{n+1}]\}$ is a L\'{e}vy process, for each $n=0, \ldots, N-1$. 
The bridge sampling algorithm is summarised in Table \ref{table:bridge}.

\begin{table}[h!]
\caption{Bridge sampling algorithm for approximating $\E[F(T,Y, Z)]$.}
\label{table:bridge}
\begin{center}
\begin{lstlisting}[mathescape,frame=single]
0. Fix $M, N\in\mbb N$ sufficiently large.
1. Sample $M$ IID copies $\xi^{(1)}, \ldots, \xi^{(M)}$ from the law of $\left(Y'_{\tau_1}, Z'_{\tau_1}, \ldots, Y'_{\tau_N}, Z'_{\tau_N}\right)$, 
2. Evaluate the estimator $\frac{1}{M}\sum_{i=1}^M \WT F^{(N)}\le(\xi^{(i)}\ri),$
with $ \WT F^{(N)}(y_0,z_0, \ldots, y_N, z_N) = \E\le[F(T, Y', Z')\bigg| 
Y'_{\tau_0} = y_0, Z'_{\tau_0} = z_0, \ldots, Y'_{\tau_N} = y_N, Z'_{\tau_N} = z_N\ri].$
\end{lstlisting}
\end{center}
\end{table}
\begin{rems}\label{rem2}\rm
The choice $N=1$ in the above algorithm 
corresponds to the case of a {\em single} large 
step bridge sampling, which is the version of the algorithm that 
was implemented to produce the results reported in 
Section~\ref{sec:NumericalResults}. 
\end{rems}

Next we focus on the application of the bridge sampling method
to the approximation of the expectation of two path-dependent functionals 
that are given in terms 
of the running maximum and the occupation time of $Y$ 
as follows:
\begin{eqnarray*}
&& F_S(T,Y, Z) := g(Y_T) I_{\le\{\ovl Y_T\leq a\ri\}}, \quad a>0,
\ \text{with}\ 
\overline{Y}_t:=\sup \{Y_s\>:\>s\leq t\},
\\
&& F_O(T,Y, Z) := \int_0^T g(Y_s)\td s, 
\end{eqnarray*}
for some function $g:\mbb R_+\to\mbb R$.
The functionals $F_S$ and $F_O$ admit the following 
multiplicative and additive decompositions 
into parts that only involve the processes $Y^{i-1,i} := \{Y_{t+\t_{i-1}}, t\in [0, \t_i-\t_{i-1}]\}$, 
for $i=1, \ldots, N$:
\begin{eqnarray*} 
&& F_S(T,Y,Z) = g(Y_T) \prod_{i=1}^N F^{(i)}_{S}(Y,Z),\qquad F^{(i)}_{S}(Y,Z ) = 
I_{\le\{\sup_{s\in[\t_{i-1}, \t_i]} Y_s \leq a\ri\}},\\
&& F_O(T,Y,Z) = \sum_{i=1}^N F_{O}^{(i)}(Y,Z), \qquad
F^{(i)}_{O}(Y,Z) = \int_{\t_{i-1}}^{\t_i} g(Y_s)\td s.
\end{eqnarray*}
These decompositions in turn imply that the conditional expectations 
\begin{eqnarray}
\label{eq:CE1}
&&\WT F^{(N)}_S(y_0,z_0, \ldots, y_N, z_N)  := 
\E\le[F_S(T, Y', Z')\bigg| 
Y'_{\tau_0} = y_0, Z'_{\tau_0} = z_0, \ldots, Y'_{\tau_N} = y_N, Z'_{\tau_N} = z_N\ri],\\
&&\WT F^{(N)}_O(y_0, z_0, \ldots, y_N, z_N) := \E\le[F_O(T, Y', Z')\bigg| Y'_{\tau_0} = y_0, Z'_{\tau_0} = z_0, \ldots, Y'_{\tau_N} = y_N, Z'_{\tau_N} = z_N\ri]
\label{eq:CE2}
\end{eqnarray}
can be expressed in terms of L\'{e}vy bridge processes, as shown next. 
\begin{Prop} 
For any $N\in\mbb N$ the following decompositions hold true:
\begin{eqnarray}
&& \WT F^{(N)}_S((y_0,z_0), \ldots, (y_N,z_N)) = g(y_N) \prod_{i=1}^N\WT F_{S}^{(i)}(y_{i-1},y_{i},z_{i-1}),\\
&& \WT F^{(N)}_O((y_0,z_0), \ldots, (y_N,z_N)) = \sum_{i=1}^N\WT F_{O}^{(i)}(y_{i-1},y_{i}, z_{i-1}),
\end{eqnarray}
where the functions $x\mapsto\WT F_{S}^{(i)}(x,y,z)$ and $x\mapsto\WT F_{O}^{(i)}(x,y,z)$
are given by
\begin{eqnarray*}
&& \WT F_{S}^{(i)}(x,y,z) = \E\le[I_{\le(\sup_{s\leq\Delta} L_s^{(0,x)\to (\Delta,y),i}\leq a\ri)}\ri],\q \WT F_{O}^{(i)}(x,y,z) = \E\le[\int_0^{\Delta}g\le(L^{(0,x)\to (\Delta,y),i}_s\ri) \td s\ri],
 \end{eqnarray*}
with $\Delta=T/N$, where $L^{(0,x)\to (\Delta,y),i}$ 
denotes the L\'{e}vy bridge process from $(0,x)$ to $(\Delta, y)$, with underlying 
L\'{e}vy process $L^{(i)}$ that is equal in law 
to $Y^{i-1,i}$ conditional on $Z_{\tau_{i-1}} = z$  and $Y_{\tau_i}=x$.
\end{Prop}
\begin{proof}The decompositions hold true as a consequence of 
the harness property of a L\'{e}vy process, the definition of a L\'{e}vy bridge 
and the fact that a L\'{e}vy process is temporally homogeneous. 
\end{proof}

\subsection{Bates-type stochastic volatility model with jumps}
By approximating the log-price process $Y$ of the Bates-type model 
by the EM scheme in \eqref{eq:Yp1}--\eqref{eq:Zp2}, 
and computing first-passage time probabilities and expected occupation times of the process $Y'$ as before using the recursive algorithm 
(as in Section~\ref{sec:NumericalResults}), we obtained the approximate values of an up-and-in call option and a range note under the Heston model and Bates-type models with double-exponential and hyper exponential jumps. We ran 
the algorithm in Table~\ref{table:bridge} with 
10 million paths ($M=10^7$) on a uniform grid $\Upsilon$ with $N=2^i$ steps for  $i=0,1,...,10$. We used the recursions with $n=7$ steps and approximated the functions $\WH F_{S}^{(i)}(x,y,z)$ by evaluating these on a grid of points and using (tri-linear) interpolation to obtain approximations of the values of the function outside the grid.
By way of comparison, we also report the results obtained by a standard (discrete-time) Euler-Maruyama approximation with 10 million paths and a varying number of (equidistant) time-steps.

\begin{table}[t!]
\begin{center}
\caption{\footnotesize Model parameters of the generalised Bates model, 
the maturity, strike, barrier and spot levels and range of 
the up-and-in call option and range note to be used in 
Figure~\ref{fig:ConvergenceSV} and Table~\ref{tab:MCbarrierSV}
(with jump-parameters as given in Table~\ref{tab:ModelParameters}).}
\label{tab:ModelParametersSV}
\small{
\begin{tabular}{lllllllcllll}
\hline
$\kappa$ & $\delta$ & $\xi$ & $\rho$ & $V_0$ & $K$ & $H$ & 
$(a_1, a_2)$ & $S_0$ & $r$ & $d$ & $T$  \\
\hline
1.0 &   0.1 &   0.2 &   -0.5 & 0.07 & 100 & 120 
& (1.15,1.35)
& 100 & 0.05 & 0.0 & 1.0 \\
\hline
\end{tabular}
}
\end{center}
\end{table}

For the results displayed in Figure \ref{fig:ConvergenceSV} we take the value corresponding to $N=1024$ as true value and compute the logarithm of the absolute errors for all other outcomes with respect to this value. 
In order to estimate the rates of decay of the error we added 
ordinary least-square regression lines to the figures.
The slopes of the OLS lines for the Heston model and the Bates-type model with double-exponential and hyper-exponential jumps that we found are $-1.03$, $-1.02$, and $-1.04$ in the case of the up-and-in call option and $-1.36$, $-0.96$ and $-1.02$,  in the case of the range note, 
which suggests a rate of decay of the error that is linear in the reciprocal of the number of steps. 

By way of comparison we also implemented the standard (discrete-time) 
Euler-Maruyama scheme for each of the three models, and found 
the corresponding three slopes of the OLS lines to be 
equal to $-0.48$ in the case of the values of 
the up-and-in call options  and to $-1.00$ in the case of values of the range notes. These results suggest that, in the case of an UIC option, 
only a square-root rate holds for the decay of the error 
as function of the reciprocal of the 
number of time-steps rather than a linear rate, which is in line 
with the well-known fact that the strong order of 
the discrete-time EM scheme is $0.5$, and that, furthermore, for 
killed diffusion models  the weak error of the discrete-time EM scheme 
has been shown to be bounded by a constant times $N^{-1/2}$
in the number of time-steps $N$ under suitable regularity assumptions 
on the coefficients and the pay-off function 
(see \cite[Thms. 2.3, 2.4]{Gobet}). 

\begin{figure}[tp!]
\centering
\mbox{
\subfigure[Up-and-in call option]{\includegraphics[scale=0.35]{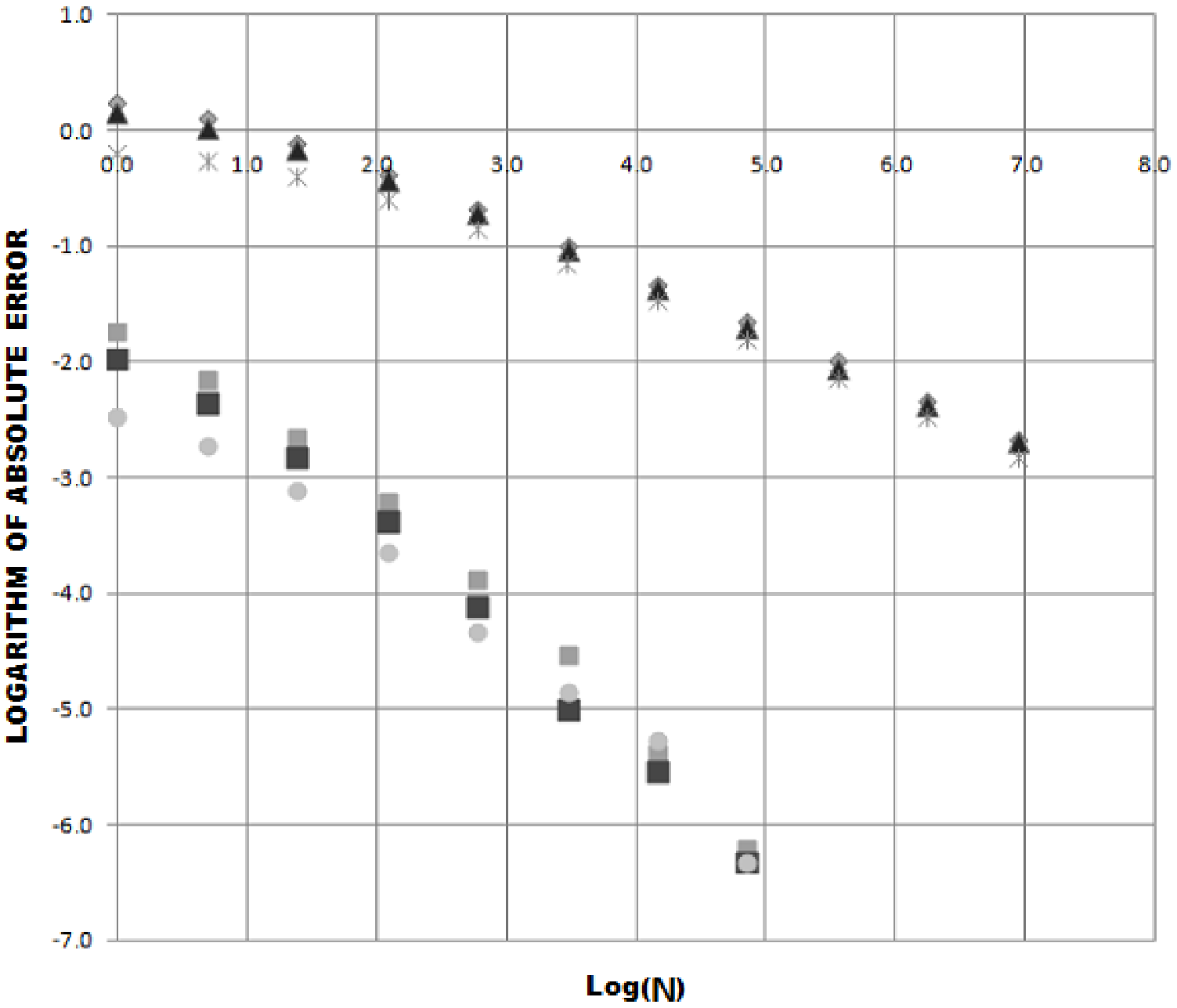}}
\subfigure[Range note]{\includegraphics[scale=0.35]{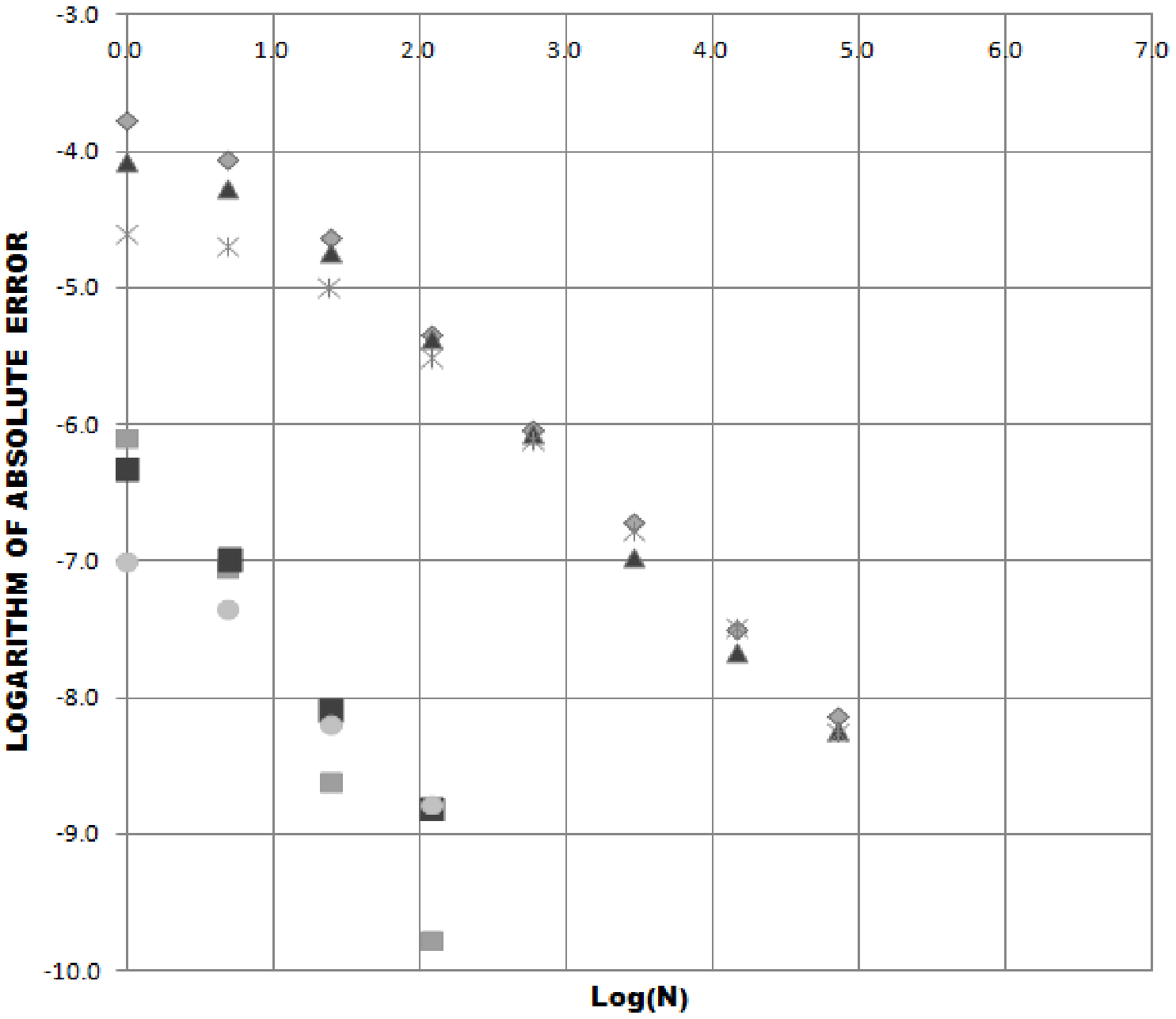}}
}
\includegraphics[scale=0.7]{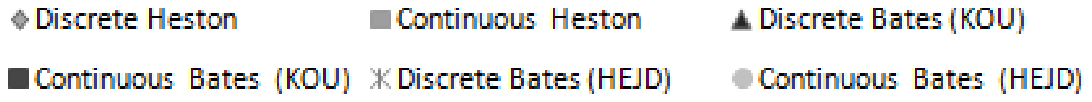}
\caption[Convergence of discrete-time and continuous-time Euler scheme]{\footnotesize The absolute error of the values of 
an up-and-in barrier option and range note under the Heston and Bates-type models plotted on a log-log scale against the number of time-steps $N$.
Parameters are as given in Tables~\ref{tab:ModelParameters} and~\ref{tab:ModelParametersSV}.}
\label{fig:ConvergenceSV}
\end{figure}
\begin{table}[ht!]
\begin{center}
\caption[Monte Carlo methods for barrier option in stochastic volatility model]{\footnotesize A comparison of different Monte Carlo methods (all ran with antithetic variates and $1$ million paths) for (i) an up-and-in call option 
and (ii) a range note. The stochastic volatility parameters and parameters of 
the derivative contracts are as given in Table~\ref{tab:ModelParametersSV}, and the jump parameters as given in Table~\ref{tab:ModelParameters}. 
In the column 'Time' the run times are reported in seconds (which include in particular the time to find the roots of the Cram\'{e}r-Lundberg equation). The continuous-time EM schemes were run using the 
first-passage time probabilities of the corresponding randomised 
bridge processes computed using $n=7$ 
recursive steps (for the barrier option) and using the expected 
occupation time 
of the corresponding randomised bridge process computed using $n=5$ 
recursive steps (for the range note). 
To obtain the values in the table marked with $^{\dag}$ and $^{*}$ we used the exact Brownian bridge probability and numerical integration,
respectively.  
}
\label{tab:MCbarrierSV}
\footnotesize{
\bgroup
\def\arraystretch{1}
\def\tabcolsep{1.7mm}
\begin{tabular}{lrrrrrrr}
\hline
&                 & \multicolumn{2}{c}{\textbf{Heston}} & \multicolumn{2}{c}{\textbf{Bates (Kou)}} & \multicolumn{2}{c}{\textbf{Bates (HEJD)}} \\
    \hline
    & Steps & Midpoint (Error) &  Time& Midpoint (Error) &  Time & Midpoint (Error) &  Time \\
\hline
\begin{tabular}{r}
Barrier option\\
\hline
\end{tabular}
\\
    Discrete-time EM & 100   &  12.755 ($\pm$0.0389) & 7.9   &  13.333 ($\pm$0.0407) & 9.1   &  15.358 ($\pm$0.0483) & 9.8 \\
    Discrete-time EM& 1,000 &  12.866 ($\pm$0.0388) & 80  &  13.432 ($\pm$0.0406) & 88  &  15.387 ($\pm$0.0481) & 94 \\
   Discrete-time EM& 10,000 &  12.935 ($\pm$0.0387) & 789 &  13.467 ($\pm$0.0406) & 888 &  15.413 ($\pm$0.0482) & 958 \\
    Continuous-time EM& 100   &  12.948 ($\pm$0.0387) & 18  &  13.468 ($\pm$0.0406) & 20  &  15.457 ($\pm$0.0482) & 82 \\
    Continuous-time EM& 1,000 &  12.956 ($\pm$0.0388) & 163 &  13.534 ($\pm$0.0408) & 165 &  15.478 ($\pm$0.0482) & 233 \\
    Continuous-time EM$^{\dag}$ & 1,000 &  12.951 ($\pm$0.0388) & 125 &       &       &       &  \\
  \hline
\begin{tabular}{r} 
Range note\phantom{ion}\\
\hline
\end{tabular}\\
    Discrete-time EM& 100   & 15.352 ($\pm$0.0373) & 8.4   & 15.374 ($\pm$0.0367) & 9.1   & 15.387 ($\pm$0.0354) & 10 \\
    Discrete-time EM& 1,000 & 15.288 ($\pm$0.0371) & 81  & 15.315 ($\pm$0.0365) & 93  & 15.309 ($\pm$0.0352) & 98 \\
    Discrete-time EM& 10,000 & 15.288 ($\pm$0.0371) & 793 & 15.304 ($\pm$0.0365) & 928 & 15.286 ($\pm$0.0351) & 1079 \\
    Continuous-time EM& 10    & 15.177 ($\pm$0.0367) & 54  & 15.237 ($\pm$0.0362) & 68  & 15.255 ($\pm$-0.035) & 132 \\
    Continuous-time EM& 100   & 15.288 ($\pm$0.0371) & 114 & 15.294 ($\pm$0.0365) & 126 & 15.327 ($\pm$0.0352) & 364 \\
    Continuous-time EM$^{*}$ & 100   & 15.288 ($\pm$0.0371) & 1491 &       &       &       &  \\
\hline
\hline
\end{tabular} 
\egroup 
}
\end{center}
\end{table}
\medskip
\appendix
\section{Proof of recursions for maxima and occupation times of a L\'{e}vy bridge}
\label{sec:BridgeSampling}
Let $X=\{X_t, t\in\mbb R_+\}$ be a L\'{e}vy process (a stochastic process with stationary and independent increments and right-continuous paths with left limits such that $X_0=0$) 
that is defined on some filtered probability space $(\Omega,\mathcal{F}, \mbf F, \P)$, 
where $\mbf F =  \{ \mathcal{F}_t, t \in\mbb R_+ \}$ denotes the completed right-continuous filtration generated by $X$. 
We refer to \cite{Kyprianou,Sato1999} for general treatments of the theory of 
L\'{e}vy processes.  To avoid degeneracies we exclude  in the sequel
the case that $|X|$ is a subordinator. 
The bridge method under consideration involves randomised bridge processes
that can informally be described as processes that are equal in law to $X$ conditioned 
to take a given value at certain independent random times. 

Formally, such a process can be constructed 
by invoking general results on existence of conditional distributions and disintegration 
(see Kallenberg~\cite[Thms. 6.3, 6.4]{Kallenberg}). 
More specifically, let the triplet $(X,\tau_1,\tau_2)$ of the L\'{e}vy process $X$ 
and independent random times $\tau_1, \tau_2$ with $\tau_1\leq\tau_2$
be defined on the Borel space $D\times U$ that is the product of
the Skorokhod space $D$ of rcll functions 
and the space $U = \mbb R^2_+$.
Then, by disintegration, we obtain a family of conditional laws conditional on different values of
 $(\eta_1,\eta_2):=(X_{\tau_1}, X_{\tau_2})$ that may be used to define the randomised bridge process with starting point 
$(\tau_1, y_1)$ and end point $(\tau_2, y_2)$ by $\{X_{(s+\tau_1)\wedge\tau_2}, s\in\mbb R_+\}$ 
for almost  all realisations $(y_1,y_2)$ of $(\eta_1,\eta_2)$. 

Under regularity assumptions on the L\'{e}vy process $X$ and for specific choices of the random times the construction in the previous paragraph may be extended to all realisations of $(\eta_1,\eta_2)$, drawing on results in \cite{ChaumontUribe} where weak-continuity results and pathwise constructions of a Markov bridges have been recently provided 
(see also~\cite{uribebravo2014} for the case of L\'{e}vy processes conditioned to stay positive).  

\begin{As}\label{as:smooth}\rm
The L\'{e}vy process $X$ satisfies the integrability condition
\begin{equation}\label{eq:Psint}
\int_{\mbb R\backslash(-1,1)} \frac{\td \theta}{|\Psi(\theta)|} < \infty,
\end{equation}
where $\Psi$ is the characteristic exponent of $X$, 
which is the function $\Psi:\mbb R\to\mbb C$ that satisfies the identity
$\E[\exp(\mbf i\theta X_t)] = \exp(-t\Psi(\theta))$ 
for all $\theta\in\mbb R$ and $t\in\mbb R_+$.
\end{As}

As random times  
we consider Gamma random variables $\Gamma_{n,q}$, $n\in\mbb N, q>0$, with mean $n/q$ 
and variance $n/q^2$ that are independent of $X$. We suppose that the pair $(X,\Gamma_{n,q})$ 
is defined on the product space $(\Omega\times\mbb R_+, \mc F\otimes\mc B(\mbb R_+), \mbf P\times P)$. 
To simplify notation we use in the sequel $\PP$ to denote the product-measure $\P\times P$.
It follows from Sato \cite[Prop. 28.1]{Sato1999} that under Assumption~\ref{as:smooth} 
the distributions under $\PP$ of both $X_{\Gamma_{n,q}}$ 
and $X_t$, $t>0$, admit continuous densities:

\begin{Lemma}
Let Assumption~\ref{as:smooth} hold. {\rm (i)} Then for any $q>0$ and $n\in\mbb N$ 
the random variable $X_{\Gamma_{n,q}}$ has a density 
$u_{n,q}$ that is continuous and bounded.

{\rm (ii)} For any $t>0$, $X_t$ admits a bounded density $p(t,x)$ that is continuous in $(t,x)\in (0,\infty)\times\mbb R$.
\end{Lemma}

Under Assumption~\ref{as:smooth} one may define the randomised L\'{e}vy bridge process starting at $(0,x)$ and pinned down at $(\Gamma_{n,q},y)$ for any $x,y\in\mbb R$. We recall first from \cite[Theorem~1]{ChaumontUribe} that, under Assumption~\ref{as:smooth}
and for any $t>0$ and $x,y\in\mbb R$ such that $p(t,y-x)>0$, there exists a Markov process on the probability space $(\Omega,\mc F,\P)$, 
denoted by $X^{(0,x)\to (t,y)}=\{X^{(0,x)\to (t,y)}_u, u\in[0,t]\}$, that starts at time $0$ at $x$ a.s., is equal to $y$ at time $t$ a.s., and 
satisfies the disintegration property. The process $X^{(0,x)\to (t,y)}=\{X^{(0,x)\to (t,y)}_u, u\in[0,t]\}$ is referred 
to as the {\em L\'{e}vy bridge process} from $(0,0)$ to $(t,y)$. 

We next specify the definition a L\'{e}vy bridge process pinned down at a Gamma random time and a given fixed end point. 
For any pair $x,y\in\mbb R$ with $u_{n,q}(y-x)>0$, the {\em randomised L\'{e}vy bridge process} 
$X^{(0,x)\to (\Gamma_{n,q},y)} = \{X^{(0,x)\to (\Gamma_{n,q},y)}_{t}, t\in\mbb R_+\}$ starting from $(0,x)$ and pinned down at $(\Gamma_{n,q},y)$ is the stochastic process with sample paths $t\mapsto\le. X^{(0,x)\to (s,y)}_{t\wedge s}(\w)\ri|_{s = \Gamma_{n,q}(\gamma)}$ for given realisations $(\w,\gamma)$
in the sample space $\Omega\times\mbb R_+$. 
The process $X^{(0,x)\to (\Gamma_{n,q},y)}$ satisfies the disintegration property (which can be shown by 
a similar line of reasoning as was given in the proof of  \cite[Theorem~1]{ChaumontUribe}), and is hence equal in law 
to the corresponding process obtained by the construction described in the second paragraph of this section. 
The derivation of the expressions for the functions $\vec D^{(1)}_q(x,y)$ and $\vec\Omega^{(1)}_q(x,y)$ is based in part 
on the following auxiliary result 
concerning the differentiability of two related functions under Assumption~\ref{as:smooth} (the proof of which is omitted as it follows by standard arguments).

\begin{Lemma}\label{lem:density} Let  Assumption~\ref{as:smooth}  hold and let $q$ be any strictly positive number.

(i) For any fixed $x\in\mbb R_+$, the function  
$
y\mapsto\PP(\overline{X}_{\Gamma_{1,q}}\leq x, X_{\Gamma_{1,q}}\leq y)$
is continuously differentiable on $\mbb R$ and its derivative $y\mapsto D_{1,q}(x,y)$ 
is bounded.

(ii) The map 
$(x,y)\mapsto\E\le[\int_0^{\Gamma_{1,q}}I_{\{X_u\leq x\}}\td u\, I_{\{ X_{\Gamma_{1,q}}\leq y\}}\ri]$ 
is continuously differentiable with respect to $x$ and~$y$ in $\mbb R$. 
The mixed derivative 
with respect to $x$ and $y$ is  given by $\Omega_{1,q}(x,y)$
for $x,y\in\mbb R$.
\end{Lemma}

The functions $D_{1,q}$ and $\Omega_{1,q}$ admit semi-analytical expressions, 
which can be derived using the Markov property and 
the Wiener-Hopf factorisation of $X$. We recall (see {\em e.g.} Bertoin~\cite[Ch. VI]{Bertoin1996}) that the probabilistic form of the Wiener-Hopf factorisation of $X$ 
states that (a) the running supremum 
$\overline{X}_{\Gamma_{1,q}}$ and 
the {\em drawdown} $\overline{X}_{\Gamma_{1,q}} - {X}_{\Gamma_{1,q}}$ 
of $X$ at the random time $\Gamma_{1,q}$ are independent, and 
(b) the drawdown $\overline{X}_{\Gamma_{1,q}} - {X}_{\Gamma_{1,q}}$ 
has the same 
law as the negative of 
the running infimum $-\unl X_{\Gamma_{1,q}}$. 
The probabilistic form of the Wiener-Hopf factorisation implies that the characteristic function
of the random variable $X_{\Gamma_{1,q}}$ 
is equal to the product of the characteristic functions $\Psi^+_q$ and $\Psi^-_q$ of $\ovl X_{\Gamma_{1,q}}$ and $\unl X_{\Gamma_{1,q}}$,
\begin{equation*}
\Psi^+_q(\theta) = \E[\exp(\mbf i\theta \overline{X}_{\Gamma_{1,q}})],\qquad 
\Psi^-_q(\theta) = \E[\exp(\mbf i\theta \underline{X}_{\Gamma_{1,q}})].
\end{equation*}

In the following result we establish that the functions 
$D_{n,q}$, $\Omega_{n,q}$ are well-defined and satisfy the recursions \eqref{eq:Drec}---\eqref{eq:Orec}:

\begin{Thm}\label{thm:rec} 
Let $q>0$, $n\in\mbb N$ and let Assumption~\ref{as:smooth} hold. 

{\rm (i)} For any $x\in\mbb R_+$, the function  $y\mapsto \PP(\ovl X_{\Gamma_{n,q}}\leq x, X_{\Gamma_{n,q}}\leq y)$
admits a continuous bounded density denoted by $D_{n,q}$. Moreover, 
the function $(x,y)\mapsto \E\le[\int_0^{\Gamma_{n,q}}I_{\{X_u\leq x\}}\td u\, I_{\{ X_{\Gamma_{n,q}}\leq y\}}\ri]$
is continuously differentiable on $\mbb R^2$ with bounded mixed-derivative denoted by $\Omega_{n,q}$.

{\rm (ii)} The functions $D_{n,q}$ and $\Omega_{n,q}$ satisfy the recursions \eqref{eq:Drec}---\eqref{eq:Orec}.
\end{Thm}
\begin{rems}{\rm Since the pinned process $X^{(0,0)\to (\Gamma_{n,q},y)}$ is equal in law to the 
process $X^{\Gamma_{n,q}} = \{X_u, u\in[0,\Gamma_{n,q}]\}$ stopped at the random time $\Gamma_{n,q}$
and conditioned on $\{X_{\Gamma_{n,q}}=y\}$, it follows that the functions $\vec D^{(n)}_q$ 
($\vec\Omega^{(n)}_q$) are equal to the ratio  of 
$D_{n,q}$ ($\Omega_{n,q}$, respectively)  and $u_{n,q}$, that is,}
\begin{eqnarray*}
D_{n,q}(x,y) = \vec D^{(n)}_q(x,y)u_{n,q}(y),\q \Omega_{n,q}(x,y) = 
\frac{\td}{\td x}\vec\Omega^{(n)}_q(x,y)u_{n,q}(y), \q 
x\in\mbb R_+,\ y\in\mbb R.
\end{eqnarray*}
\end{rems}
\begin{proof}[Proof of Theorem~\ref{thm:rec}] 
(i)  Several applications of the strong Markov property of $X$ 
and the lack of memory property of the exponential distribution
yield
\begin{eqnarray*} \PP\left[\overline{X}_{\Gamma_{n,q}}\leq x, X_{\Gamma_{n,q}}\in\td y \right]
&=& \PP\left[\tau^+_x \ge \Gamma_{n,q}, X_{\Gamma_{n,q}}\in\td y\right]\\
&=& \PP\left[X_{\Gamma_{n,q}}\in\td y\right] - \sum_{k=1}^n\PP\left[\Gamma_{k-1,q}\leq \tau^+_x  < \Gamma_{k,q}, X_{\Gamma_{n,q}}\in\td y\right]\\
&=& \PP\left[X_{\Gamma_{n,q}}\in\td y\right] - \sum_{k=1}^n \int_{\mbb R_+}
\E\left[I_{\{ \Gamma_{k-1,q}\leq \tau^+_x < \Gamma_{k,q}\}}I_{\{X_{\tau^+_x}\in\td z\}}\right]
 \PP[z+X_{\Gamma_{n-k+1,q}}\in\td y ],
\end{eqnarray*}
with $\Gamma_{0,q}:=0$.
Taking the Fourier transform of the measure $r^{n,q}_x(\td y) := \PP\left[\overline{X}_{\Gamma_{n,q}}\leq x, X_{\Gamma_{n,q}}\in\td y \right]$ we find
\begin{equation}\label{eq:s}
\mc F r_x(s) = 
\E[\exp\{\mbf i s X_{\Gamma_{n,q}}\}] - \sum_{k=1}^n
\E[\exp\{\mbf i s X_{\Gamma_{n-k+1,q}}\}]
\E[\exp\{\mbf i s X_{\tau^+_x}\}I_{\{ \Gamma_{k-1,q}\leq \tau^+_x < \Gamma_{k,q}\}}], \q s\in\mbb R.  
 \end{equation}
Since the second factors in the sum in \eqref{eq:s} 
are bounded by one and 
\begin{equation}\label{eq:FFun}
\E[\exp(\mbf i\theta X_{\Gamma_{n,q}})] = \left(\frac{q}{q + \Psi(\theta)}\right)^n, 
\end{equation} we have  $|\mc F r_x(s)|\leq \sum_{k=1}^n\int q^k|q + \Psi(s)|^{-k}\td s$, 
for any $x\in\mbb R_+$, $q>0$ and $n\in\mbb N$, 
which is finite by Assumption~\ref{as:smooth} and the bound $|q/(q+\Psi(s))|\leq 1$ that holds for all $s\in\mbb R$. 
We conclude that, for any $x\in\mbb R_+$, the measure $r^{n,q}_x(\td y)$ admits a continuous bounded density (by Sato~\cite[Prop. 28.1]{Sato1999}).

We show the required differentiability of $\E\le[\int_0^{\Gamma_{n,q}}I_{\{X_u\leq x\}}\td u\, I_{\{ X_{\Gamma_{n,q}}\leq y\}}\ri]$ 
by induction with respect to $n$. Noting that the case $n=1$ follows from Lemma~\ref{lem:density}(ii), we next turn 
to the induction step. Assume thus that the assertion is valid for given $n\in\mbb N$. We have by an application of the Markov property \begin{eqnarray}\label{eq:decoc}
&& \E\left[\int_0^{t+u} I_{\{X_s\leq x\}}\td s\>\> I_{\{X_{t+u}\in \td b\}} \right] =
\int_{w\in\mbb R} \E \left[\int_0^{t} I_{\{X_s\leq x\}}\td s\>\> I_{\{X_{t}\in \td w\}}\right]\PP \left[w+ X_{u}\in \td b\right]  \\
&& \q\q+ \int_{w\in\mbb R} \E\left[\int_0^{u} I_{\{w+X_s\leq x\}}\td s\>\> I_{\{w+X_{u}\in \td b\}} \right]
\PP \left[X_t\in \td w\right], 
\nonumber 
\end{eqnarray}
for any real $x$.  Replacing in \eqref{eq:decoc}
$t$ and $u$ by the independent random times 
$\Gamma_{1,q}$ and $\Gamma_{n-1,q}$, using the fact that their sum is equal in distribution to $\Gamma_{n,q}$ 
and that the random variables 
$X_{\Gamma_{n,q}}$ and $X_{\Gamma_{1,q}}$ have continuous densities 
$u_{n,q}$ and $u_{1,q}$ (by Lemma~\ref{lem:density}), it follows from the induction assumption that the assertion is valid 
for $n+1$. It follows thus by induction that we have the required differentiability for all $n\in\mbb N$.

(ii)  Since we may write
$$\ovl X_t = \max\le\{X_s + \sup_{0\leq u\leq t-s}(X_{u+s}-X_s), \ovl X_s\ri\},
\q\text{for any $s,t$ with $0\leq s\leq t,$}
$$
it follows as a consequence of the stationarity and independence of increments of $X$, and the fact that a $\Gamma_{n,q}$ random variable is equal in distribution to the sum 
of independent $\Gamma_{n-1,q}$ and $\Gamma_{1,q}$ random variables
that we have
\begin{eqnarray}
\label{eq:RecursiveFormula} \lefteqn{ 
\PP\left(\ovl X_{\Gamma_{n,q}} \leq x, X_{\Gamma_{n,q}}\in\td w\right) 
= \PP\le(\max\le\{X_{\Gamma_{1,q}}
+\ovl X'_{\Gamma_{n-1,q}},\ovl X_{\Gamma_{1,q}}\ri\} 
\leq x, 
X_{\Gamma_{1,q}} + X'_{\Gamma_{n-1,q}}\in\td w\ri)}  \\
&=& \int_{(- \infty,x]}  \PP(\ovl X_{{\Gamma_{1,q}}} \leq x,X_{{\Gamma_{1,q}}}\in\td z)
\PP(z+\ovl X_{\Gamma_{n-1,q}} \leq x, z+ X_{{\Gamma_{n-1,q}}}\in\td w),
\nonumber
\end{eqnarray}
where the random variables  $\ovl X'_{\Gamma_{n-1,q}}$ 
and $X'_{\Gamma_{n-1,q}}$ 
are independent of $X$. We arrive at 
the identity in \eqref{eq:Drec} 
since the L\'{e}vy process $X$ is spatially homogeneous.

The recursion follows from \eqref{eq:decoc} replacing as before
$t$ and $u$ by the independent random times 
$\Gamma_{1,q}$ and $\Gamma_{n-1,q}$ and using the fact that their sum is equal in distribution to a $\Gamma_{n,q}$ random variable.
\end{proof}
 
\end{document}